\documentclass[11pt]{amsart}
\usepackage[latin1]{inputenc}
\usepackage{amssymb}
\usepackage{amsmath}
\usepackage{amsfonts}
\usepackage{amsthm}
\usepackage{enumerate}
\usepackage{color}

\newcommand{\R}{\mathbb{R}}

\newcommand{\N}{\mathbb{N}}
\newcommand{\T}{\mathbb{T}}
\newcommand{\Z}{\mathbb{Z}}
\newcommand{\be}{\begin{equation}}
\newcommand{\ee}{\end{equation}}

\newcommand{\re}[1]{(\ref{#1})}
\newtheorem{thm}{Theorem}[section]
\newtheorem{lem}[thm]{Lemma}
\newtheorem{coro}[thm]{Corollary}
\newtheorem{pro}[thm]{Proposition}
\newtheorem{rem}[thm]{Remark}

\newtheorem*{merci}{Acknowledgements}

\begin{document}
\title[ KP-II on the background of a non localized solution ]
{\bf Global well-posedness for the KP-II equation on the background of a non localized solution}
\author{Luc Molinet} 
\address{Laboratoire de Math\' ematiques et Physique Th\' eorique, UMR CNRS 6083, Facult\' e des Sciences et Techniques, Universit\' e Fran\c{c}ois Rabelais, Parc de Grandmont, 37200 Tours Cedex}
\author{Jean-Claude Saut}
\address{ Universit\'e de Paris-Sud et CNRS, UMR de Math\'ematiques, B\^at. 425, 91405 Orsay Cedex }
\author {Nikolay Tzvetkov}
\address
{D\'epartement de Math\'ematiques, Universit\' e de Cergy-Pontoise, 2 avenue Adolphe Chauvin, 95302 Cergy-Pontoise Cedex}
\email{luc.molinet@lmpt.univ-tours.fr,  jean-claude.saut@math.u-psud.fr, nikolay.tzvetkov@u-cergy.fr}
\date{October 18, 2010}
\begin{abstract}
Motivated by transverse stability issues, we address  the time evolution under the KP-II  flow of perturbations of a solution
which does not decay in all directions, for instance the KdV-line soliton. We study two
different types of perturbations : perturbations that are square integrable in $ \R\times \T $ and perturbations that are square
integrable in $ \R^2 $. In both cases we prove the global well-posedness of the Cauchy problem associated with such initial
data.
\end{abstract}
\subjclass[2000]{}
\date{}
\maketitle
\section{Introduction}
\subsection{Presentation of the problem}
The  Kadomtsev-Petviashvili (KP) equations
\begin{equation}\label{KP}
(u_t+u_{xxx} +u u_x )_x \pm u_{yy} =0
\end{equation}
were introduced in \cite{KaPe} to study the transverse stability of the solitary wave solution of the 
Korteweg- de Vries equation
\begin{equation}\label {KdV}
u_t+u_x+uu_x+(T-\frac{1}{3})u_{xxx}=0, \;x\in \R,\;t\in\R.
\end{equation}
Here $T\geq 0$ is the Bond number which measures  surface tension effects in the context of surface hydrodynamical waves.
Actually the (formal) analysis in \cite{KaPe} consists in looking for a {\it  weakly transverse} perturbation 
of the transport equation
$u_t+u_x=0.$
This perturbation amounts to adding  to the equation a nonlocal term, 
namely $\frac{1}{2}\partial_x^{-1}u_{yy}$.
The same formal procedure is applied to the KdV equation \eqref{KdV} yielding the KP equation of the form
\begin{equation}\label{KPbrut}
u_t+u_x+uu_x+(T-\frac{1}{3})u_{xxx}+\frac{1}{2}\partial_x^{-1}u_{yy}=0.
\end{equation}
By change of frame and scaling, \eqref{KPbrut} reduces to \eqref{KP} with the $+$ sign (KP-II) when $T<\frac{1}{3}$ and the $-$ sign (KP-I) when $T>\frac{1}{3}$.

As far as the {\it transverse stability} of the KdV solitary wave  ("1-soliton")  $\psi_c(x-ct,y)$, where
\begin{equation}\label{KdV_onde_solitaire}
\psi_c(x,y)=\frac{3c}{2}\,{\rm cosh}^{-2}\Big(\frac{\sqrt{c}\, x}{2}\Big),
\end{equation}
is concerned, the natural initial condition 
associated to \eqref {KP} should be $u_0=\psi_c+v_0$ where $v_0$ is either "localized" in $x$ and $y$, or localized in $x$ and $y$-periodic.
In any case this rules out initial data in Sobolev spaces like $H^s(\R^2)$ or their anisotropic versions $H^{s_1,s_2}(\R^2)$, as was for instance the case considered in \cite {Bo3}, \cite {TT}, \cite {Tz}, 
\cite {H}, \cite{HHK} for the KP-II equation or in \cite{MST}, \cite{K}, \cite{IKT} for the KP-I equation.

Actually, it was proved in  \cite{MST-ter} that the Cauchy problem for the KP-I equation is globally well-posed for data which are localized perturbations of arbitrary size of a {\it non-localized} traveling wave solution such as the KdV N-soliton or the Zaitsev soliton (which is a localized in $x$ and periodic in $y$ solitary wave of the KP-I equation). The same result has been proven in \cite{FP} for localized perturbations of {\it small} size  using Inverse Scattering techniques.

No such result seems to be known for the KP-II equation for data of arbitrary size. The aim of the present paper is to fix this issue. 
Observe that the results in \cite{GPS} concern initial data localized in $y$ and periodic in $x$, for instance  
belonging to $H^s(\T\times \R)$ which excludes initial data of type $\psi_c+v_0$ as above.

On the other hand, the Inverse Scattering method has been used formally in \cite{AV} and rigorously in \cite {AV2} to study the Cauchy problem for the KP-II equation with nondecaying data along a line, that is $u(0,x,y)= u_{\infty}(x-vy)+ \phi(x,y)$ with $\phi(x,y)\to 0$ as $x^2+y^2\to \infty$ and $u_{\infty}(x)\to 0 $ as $|x|\to \infty$.  Typically, $u_{\infty}$ is the profile of a traveling wave solution with its peak localized on a moving line in the $(x,y)$ plane.
It is a particular case of the $N$- soliton of the KP-II equation discovered by Satsuma \cite{Sat} (see the derivation and the explicit form when $N=1,2$ in the Appendix of \cite{NMPZ}). As in all results obtained for KP equations by using the Inverse Scattering method, the initial perturbation of the non-decaying solution is supposed to be small enough in a weighted $L^1$ space (see \cite{AV2} Theorem 13).

As will be proven in the present paper, PDE techniques allow to consider {\it arbitrary large perturbations} of a (smaller) class of non-decaying  solutions of the KP-II equation. 

We will therefore  study here  the initial value problem for the Kadomtsev-Petviashvili (KP-II) equation
\begin{equation}\label{KP2}
(u_t+u_{xxx} +u u_x )_x + u_{yy} =0\,.
\end{equation}
We will either consider periodic in $y$ solutions or we will suppose that
$u=u(t,x,y)$, $(x,y)\in\R^2$, $t\in\R$, with initial data
\begin{equation}\label{1.2}
u(0,x,y)=\phi(x,y)+\psi(x,y),
\end{equation}
where $\psi$ is the profile of a non-localized (i.e. not
decaying in all spatial directions) traveling wave of the KP-II
equation and $\phi$ is localized, i.e. belonging to Sobolev spaces on $\R^2$.
We recall (see \cite {deBS}) that, contrary to the KP-I equation, the KP-II equation does not possess any {\it localized in both directions} traveling wave solution. The background solution  could be for instance the line soliton (1-soliton)   $\psi_c(x-ct,y)$ moving with velocity $c$ of the Korteweg-
de Vries (KdV) equation defined by \eqref{KdV_onde_solitaire}.
But it could also be the profile of the N-soliton solution of the KdV equation, $N\geq 2.$
The KdV N-soliton is of course considered as a two dimensional (constant
in $y$) object. 

Solving the Cauchy  problem in both those functional settings can be viewed as a preliminary (and necessary !) step towards the study of the dynamics of the KP-II equation on the background of a non fully localized solution, in particular towards  a PDE proof of the nonlinear   stability of the KdV soliton or N-soliton  with respect to transversal perturbations governed by the KP-II flow.  This has been established in \cite{AV2} Proposition 17 by Inverse Scattering methods and very recently by Mizumachi and Tzvetkov \cite{MT} who proved (by PDE techniques but using the Miura transform for the KP II equation) the $L^2$ stability of the KP II line soliton with respect to transverse perturbations.

 We recall that it is established in \cite{Za} by Inverse Scattering methods and in \cite{RT,RT2} by PDE techniques which extend to a large class of (non integrable) problems, that the KdV 1-soliton is transversally nonlinearly unstable with respect to the KP I equation.

 The advantage of the PDE approach of the present paper compared to an Inverse Scattering one is that it can be straightforwardly applied to non integrable equations such as  the higher order KP-II equations (see \cite{ST2}, \cite{ST3}).

\subsection{Statement of the results}
As was previously mentioned our main result is that the KP-II equation is globally well-posed for data of arbitrary size in the two afterementioned functional settings.
\begin{thm}\label{theo1}
The Cauchy problem associated with the KP-II equation is globally well-posed in
 $ H^s(\R\times\T) $ for any $ s\ge 0$. More precisely for every $\phi\in H^s(\R\times \T)$ there is a global strong solution of
 \eqref{KP2} in $C(\R;H^s(\R\times\T))$ with initial datum $u(0,x,y)=\phi(x,y)$. The solution is unique in $C(\R;H^s(\R\times\T))$ if
 $s>2$. For $s\in [0,2]$ the uniqueness holds in the following sense. For every $T>0$  there is a Banach space $X_T$ continuously embedded in
 $C([-T,T];H^s(\R\times\T))$ and containing $C([-T,T];H^{\infty}(\R\times\T))$ such that the solution $u$ belongs to $X_T$ and is unique in this class. Moreover the flow is Lipschitz continuous on bounded sets of $ H^s(\R\times\T) $. Namely, for every bounded set $B$ of 
 $ H^s(\R\times\T) $ and every $T>0$  there exists a constant $C$ such that for every $\phi_1, \phi_2\in B$ the corresponding solutions $u_1$, $u_2$
 satisfy
 $$
 \|u_1-u_2\|_{L^\infty([-T,T];H^s(\R\times\T))}\leq C\|\phi_1-\phi_2\|_{H^s(\R\times\T)}\,.
 $$ 
 Finally the $L^2$ norm is conserved by the flow, i.e.
 $$
 \|u(t,\cdot)\|_{L^2(\R\times\T)}=\|\phi\|_{L^2(\R\times\T)}\,,\qquad \forall\, t\in\R.
 $$

 \end{thm}
We next state our result concerning localized perturbations.
\begin{thm}\label{theo2}
Let $\psi_c(x-ct,y)$ be a solution of the KP-II equation such that for every
 $\sigma \ge 0 $,
$
(1-\partial_x^2 -\partial_y^2)^{\frac{\sigma}{2}}\psi_c\,:\, \R^2\rightarrow \R
$
is bounded and belongs to $ L^2_x L^\infty_y (\R^2)$.
Let $s\geq 0$ be an integer.
Then for every $\phi\in H^s(\R^2) $ there exists a global strong solution
$u$ of \re{KP2} with initial data $\phi+\psi_c$. The uniqueness holds in the following sense. For every $T>0$
 there is a Banach space $X_T$ continuously embedded in
 $C([-T,T];H^s(\R\times\T))$ and containing $C([-T,T];H^{\infty}(\R\times\T))$ such that the solution $u$ belongs to $X_T$ and is unique in this class
$
\psi_c(x-ct,y)+X_{T}.
$
\end{thm}
The spaces $X_T$ involved in the statements of the above results are suitable Bourgain spaces defined below.

Our proof is based on the approach by Bourgain to study the fully periodic case. We need however to deal with  difficulties linked to several low frequencies cases 
(see for example Lemma~\ref{lemme3} below) which do not occur in the purely periodic setting. Moreover in the proof of Theorem~\ref{theo2} one needs to
make a use of the Kato type smoothing effect for KP-II which was not present in previous works on the low regularity well-posedness of the KP-II equation.

\begin{rem}\label{Nsol}
As was previously noticed, the result of Theorem \ref{theo2} holds (with the same proof) when we replace  $\psi_c$  by the value at $t=0$ of the N-soliton solution $S_N$ of the KdV equation which satisfies similar smoothness and decay properties as $\psi_c$. This follows from the structure of the $S_N$. For instance (see \cite{DJ}) the 2-soliton of the KdV equation written on the form
$$u_t-6uu_x+u_{xxx}=0$$
is 
$$
S_2(x,t)=-12\frac{3+4\cosh (2x-8t)+\cosh (4x-64t)}
{\big(3\cosh (x-28t)+\cosh (3x-36t)\big)^2}.
$$
On the other hand, one cannot take instead of  $\psi_c$  a  function $\psi$ which is non-decaying along a line $\lbrace (x,y)\vert x-vy=x_0\rbrace,$ as for instance the line-soliton of the KP-II equation which writes $\psi (x-vy-c t).$ However, the change of variables  
$(X=x-vy, Y=y)$ transforms the KP-II equation into 
$$u_t-2vu_Y+v^2u_X+u_{XXX}+uu_X+\partial_X^{-1}u_{YY}=0$$
and probably our analysis applies to this equation with an initial data  which is  a localized (in $(X,Y)$) perturbation of  $\psi(X)$, at least for suitable values of $(c,v)$.
\end{rem}
\subsection{Organization of the paper}
The second Section is devoted to the proof of a bilinear estimate for localized functions in $\R^2\times \Z.$ It is based on Bourgain's method in \cite{Bo3} who considers the case of functions in $\R\times\Z^2.$ We prove Theorem \ref{theo1} in Section 3 as a consequence of two bilinear estimates in Bourgain $X^{b,s}$ spaces. Finally we prove Theorem \ref{theo2} in Section 4 by a fixed point argument in suitable Bourgain spaces,  using in a crucial way the dispersive and smoothing properties of the KP-II linear group.
\subsection{Notations}
We will denote $\|\cdot\|_{L^p}$ (resp $\|\cdot\|_{H^s}$) the standard  
norm in the Lebesgue space $L^p$ (resp. the Sobolev space $H^s$), the domain being clear from the context. 
The Fourier transform on $\R^2_{x,y}\times\R_t$ (resp. $\R_x$) is denoted $\mathcal F$ or \;$\widehat{}$\; (resp. $\mathcal F_x$).
We will use the "Japanese bracket notation" $\langle \cdot\rangle =(1+|\cdot|^2)^{1/2}.$
The notation $U(t)$ will stand for the (unitary in all $H^s(\R^2)$ or $H^s(\R\times \T)$ Sobolev spaces) KP-II group, 
that is $U(t)= e^{-t(\partial_x^3+\partial_x^{-1}\partial_{y}^2)}$.
For $(b,s)\in \R\times \R,$ the norm of the Bourgain space associated to $U(t)$ is, for functions defined on $\R^2_{x,y}\times \R_t$ (with the obvious  modification that integration in $\eta \in \R$ is replaced by summation in $q\in \Z$ for functions defined on $\R_x\times\T_y\times\R_t$):
$ \|u\|_{X^{b,b_1,s}}=\|U(-t)u\|_{H^{b,b_1,s}},$ where 
$$
\|u\|^2_{H^{b,b_1,s}}=\int _{\R_{\xi,\eta}^2\times\R_{\tau}}
\Big\langle\frac{\langle \tau  \rangle}{\langle \xi \rangle^3}\Big\rangle^{2b_1}
\langle\zeta\rangle^{2s} \langle \tau\rangle^{2b} |\hat{u}(\tau, \xi,\eta)|^2d\tau d\xi d\eta,\; 
\zeta=(\xi,\eta)$$
that is 
$$\|u\|^2_{X^{b,b_1,s}}=
 \int _{\R_{\xi,\eta}^2\times\R_{\tau}}
\Big\langle\frac{\langle \sigma(\tau,\xi,\eta)  \rangle}{\langle \xi \rangle^3}\Big\rangle^{2b_1}
 \langle\zeta\rangle^{2s}\langle\sigma(\tau,\xi,\eta)\rangle^{2b}  |\hat{u}(\tau, \xi,\eta)|^2d\tau d\xi d\eta,
$$
where $\sigma(\tau,\xi, \eta)= \tau-\xi^3+\eta^2/\xi$.
For any $T>0,$  the norm in the localized version  $X^{b,b_1,s}_T$ is defined as 
$$
\|u\|_{X^{b,b_1,s}_T}= \inf \lbrace \|w\|_{X^{b,b_1,s}},\;w(t)=u(t) \;\text{on}\; (-T, T)\rbrace.
$$
For $(b,s)\in \R\times\R$ the space $Z^{b,s}$ is the space associated to the norm
$$
\|u\|_{Z^{b,s}}=\|\langle \sigma\rangle^{b-\frac{1}{2}}\langle\zeta\rangle^s\widehat{u}\|_{L^2_{\zeta}L^1_{\tau}}.
$$
We define the restricted spaces $Z^{b,s}_T$ similarly to above.
The notation $\lesssim$ is used for $\leq C$ where $C$ is  an irrelevant constant.  For a real number $s,$ $s-$ means "for  any $r<s$ close enough to $s$". 
If $A, B\in \R,$ we denote $A\wedge B =\min (A,B)$ and $A\vee B=\max (A,B).$ $A\sim B$ means that there exists $c>0$ such that $\frac{1}{c}|A|\leq |B|\leq c|A|.$
If $S$ is a Lebesgue measurable set in $\R^n$, $|S|,$  or ${\rm mes}\; S,$ denotes its Lebesgue measure. 
$\#S$ denotes the cardinal of a finite set $S$.
\section{A  bilinear estimate for localized functions in $ \R^2\times\Z$}
In this section we will prove the following crucial bilinear estimate for functions having some  localizations related to the KP-II dispersion relation 
in $ \R^2\times \Z $. 
This is essentially a bilinear $L^4-L^2$ Strichartz estimate associated to the linear KP-II group.
We mainly follow the proof of Bourgain \cite{Bo3} that treats the case of functions in $\R\times \Z^2 $.
 \begin{pro} \label{propo1}
Let $ u_1 $ and $ u_2 $ be two real valued $ L^2 $ functions defined on $ \R\times \R\times \Z $ with the following support properties
\begin{equation*}
(\tau,\xi,q)\in {\rm supp}\, (u_i) \Rightarrow \xi\in I_i , \, |\xi|\sim M_i ,
\langle \tau-\xi^3+q^2/\xi \rangle \sim K_i , \, i=1,2.
\end{equation*}
Then the following estimates hold ,
\begin{multline} \label{pro1est1}
\| u_1\star u_2\|_{L^2} \lesssim  (K_1\wedge K_2)^{1/2}(|I_1|\wedge |I_2|)^{1/2}
\\
\langle(K_1\vee K_2)^{1/4} (M_1\wedge M_2)^{1/4}\rangle \| u_1\|_{L^2}  \|  u_2\|_{L^2}
\end{multline}
and if $ M_1\wedge M_2 \ge 1 $  then
\begin{multline}\label{pro1est2}
  \| u_1\star u_2\|_{L^2} \lesssim (K_1\wedge K_2)^{1/2}(K_1\vee K_2)^{1/4}
   (M_1\wedge M_2)^{1/4}
   \\
  \Bigl[ (M_1\wedge M_2)^{1/4} \left\langle
  \frac{(K_1\vee K_2) }{(M_1\vee M_2)^{1-} }\right\rangle
  +\frac{ (|I_1|\wedge |I_2|)^{1/2}}{(K_1\vee K_2)^{1/4}}\Bigr]
  \| u_1\|_{L^2}   \| u_2\|_{L^2} \, .
\end{multline}
\end{pro}
 \subsection{Three basic lemmas}
 Before starting the proof of the proposition let us recall the three following basic lemmas that we will use intensively.
\begin{lem}\label{le1}
Consider a set $ \Lambda \subset \R_\xi\times \Z_q $. Let the
projection of $ \Lambda $ on the $ \xi $ axis be contained in a
set $ I\subset \R $. Assume in addition that for any fixed $ \xi_0
\in I $ the cardinal of the set $ \Lambda\cap \{(\xi_0,q), q\in
\Z\} $ is bounded by a constant $ C $. Then $ |\Lambda |\le C |I
|$.
\end{lem}
\begin{lem}\label{le2}
Let $ I $ and $ J $ be two intervals on the real line and $ f \,
:\, J\mapsto \R $ be a smooth function. Then
$$
{\rm mes} \{x\in J \, :\, f(x)\in I \} \le \frac{|I|}{\inf_{\xi\in J}
|f'(\xi)|}
$$
and
$$
\#\{q\in J\cap \Z \, :\, f(q)\in I \} \le \frac{|I|}{\inf_{\xi\in
J} |f'(\xi)|}+1\; .
$$
\end{lem}
\begin{lem}\label{le3}
Let $ a\neq 0 $, $b$, $c$ be real numbers and $ I $ be an
interval on the real line. Then
$$
{\rm mes} \{x\in \R \, : \, ax^2+bx +c\in I \} \le
\frac{|I|^{1/2}}{|a|^{1/2}}
$$
and
$$
\#\{q\in  \Z \, :\, \, aq^2+bq +c\in I \} \le
\frac{|I|^{1/2}}{|a|^{1/2}}+1 \; .
$$
\end{lem}
\subsection{Proof of  Proposition \ref{propo1}}
 According to [\cite{Bo3},p. 320], we can suppose that $
\xi \ge 0 $  on the support of $ u_j(\tau,\xi,q) $ (see also
[\cite{ST},p. 460] for a detailed discussion). We thus have to
evaluate
$$
\sum_{q\in\Z} \int_{\R\times \R_+} \Bigl| \sum_{q_1\in \Z}
\int_{\R\times \R_+} \chi_{\{\xi-\xi_1\ge 0\}} u_1(\tau_1, \xi_1,
q_1) u_2(\tau-\tau_1,\xi-\xi_1,q-q_1) d\tau_1 d\xi_1 \Bigr|^2
d\tau\, d\xi
$$
By Cauchy-Schwarz inequality in $ (\tau_1,\xi_1,q_1) $ we thus
get
\begin{equation} \label{eq8}
\| u_1\star u_2\|_{L^2}^2 \lesssim \sup_{(\tau,\xi\ge 0,\eta)}
|A_{\tau,\xi,q}| \|u_1\|_{L^2}^2 \|u_2\|_{L^2}^2
\end{equation}
where $ A_{\tau,\xi,q} \subset \R^2\times \Z $ is the set
 \begin{eqnarray*}
 A_{\tau,\xi,q} &:= & \Bigl\{(\tau_1,\xi_1,q_1): \xi_1\in I_1, \xi-\xi_1\in I_2, \, 
 0<\xi_1 \sim M_1, \, 0<\xi-\xi_1\sim M_2  \\
  & &
 \Bigl\langle \tau_1-\xi_1^3+\frac{q_1^2}{\xi_1}
  \Bigr\rangle\sim K_1,   \Bigl\langle
  \tau-\tau_1-(\xi-\xi_1)^3+\frac{(q-q_1)^2}{\xi-\xi_1} \Bigr\rangle\sim
  K_2\Bigr\} \; .
 \end{eqnarray*}
 A  use of the triangle inequality yields
 \begin{equation} \label{eq9}
 |A_{\tau,\xi,q} |\lesssim (K_1\wedge K_2) |B_{\tau,\xi,q}|
 \end{equation}
 where
\begin{eqnarray*}
 B_{\tau,\xi,q} &:= & \Bigl\{(\xi_1,q_1):  \xi_1\in I_1, \xi-\xi_1\in I_2, \, 
 0<\xi_1 \sim M_1, \, 0<\xi-\xi_1\sim M_2  \\
  & &
 \Bigl\langle \tau-\xi^3+\frac{q^2}{\xi} +3\xi \xi_1 (\xi-\xi_1)
 +\frac{(\xi q_1-\xi_1 q)^2}{\xi \xi_1(\xi-\xi_1)}
  \Bigr\rangle\lesssim (K_1\vee K_2)  \Bigr\} \; .
 \end{eqnarray*}
 Let us  now first prove (\ref{pro1est1}) by bounding
$|B_{\tau,\xi,q}|$ in a direct way. The measure of the projection
of $B_{\tau,\xi,q} $ on the $ \xi_1 $-axis is bounded by
 $ (|I_1|\wedge |I_2) $ and for a fixed $ \xi_1 $, using Lemma \ref{le3}, the cardinal of its  $ q_1 $-section is bounded by
 $ c(M_1\wedge M_2)^{1/2} (K_1\vee K_2)^{1/2}+1$.  Therefore Lemma \ref{le1} and (\ref{eq8})-(\ref{eq9}) yield the bound \eqref{pro1est1}.
 
To prove (\ref{pro1est2}) we  divide $ B_{\tau,\xi,q} $ into two regions by setting
\begin{equation}
B^1_{\tau,\xi,q}:=\{(\xi_1,q_1)\in B_{\tau,\xi,q}  \,:\,
\Big|\frac{q_1}{\xi_1}-\frac{q-q_1}{\xi-\xi_1}\Big|\le 1  \}
\end{equation}
and
\begin{equation}
B^2_{\tau,\xi,q}:=\{(\xi_1,q_1)\in B_{\tau,\xi,q}  \,:\,
\Big|\frac{q_1}{\xi_1}-\frac{q-q_1}{\xi-\xi_1}\Big|\ge 1  \}
\end{equation}
In the next lemma we estimate the size of the first region.
\begin{lem}\label{lem1}
For $M_1\wedge M_2\ge 1$, the following estimate holds 
\begin{equation*}
|B^1_{\tau,\xi,q}| \lesssim (K_1\vee K_2)^{1/2} (M_1\wedge M_2)^{1/2} \,.
\end{equation*}
\end{lem}
\begin{proof}
Recall that $ \xi_1 $ and $ \xi-\xi_1 $ are  non
negative real numbers and thus $ \xi \sim \xi_1\vee( \xi-\xi_1)$. Note also that
\begin{equation}\label{uio}
\frac{(\xi   q_1 -\xi_1 q)^2}{\xi \xi_1 (\xi-\xi_1)}=\frac{\xi_1 (\xi-\xi_1)}{\xi} \Bigl( \frac{q_1}{\xi_1}-\frac{q-q_1}{\xi-\xi_1}\Bigr)^2
\end{equation}
 From the definition of   $ B^1_{\tau,\xi,q} $  we  thus deduce that
\begin{equation*}
|\tau-\xi^3+q^2/\xi+3 \xi \xi_1 (\xi-\xi_1)| \lesssim K_1\vee
K_2+M_1\wedge M_2 \; .
 \end{equation*}
According to Lemma \ref{le3},  the projection  of  $
B^1_{\tau,\xi,q} $ on the $ \xi_1 $-axis is thus bounded by
\begin{equation}\label{lem1proof2}
\frac{\Bigl(  K_1\vee K_2+M_1\wedge M_2 \Bigr)^{1/2}}{(M_1\vee M_2)^{1/2}} \lesssim
 \frac{(  K_1\vee K_2)^{1/2}}{(M_1\vee M_2)^{1/2}}+
 \frac{(  M_1\wedge  M_2)^{1/2}}{(M_1\vee M_2)^{1/2}}
 \end{equation}
 We separate two cases : \\
 {\bf 1)} $ M_1\wedge M_2\le K_1\vee K_2 $. Then  from  the definition of $ B^1_{\tau,\xi,q} $ and Lemma \ref{le2} we infer
  that for   $ \xi_1 $ fixed, the cardinal  of the $ q_1 $ section of  $ B^1_{\tau,\xi,q} $ is bounded from above by $ \langle
  M_1\wedge M_2 \rangle \sim M_1\wedge M_2$.   We thus get by Lemma~\ref{le1} that in the considered case
  $$
 | B^1_{\tau,\xi,q}| \lesssim  \frac{(  K_1\vee K_2)^{1/2}( M_1\wedge  M_2) }{(M_1\vee M_2)^{1/2}}
 \leq  (K_1\vee K_2)^{1/2} (M_1\wedge M_2)^{1/2} \; .
 $$
 {\bf 2)} $M_1\wedge M_2\ge K_1\vee K_2 $. In this case we subdivide once more :
 \begin{enumerate}
 \item [2.1)]$ \; \Bigl| \frac{q_1}{\xi_1} -\frac{q-q_1}{\xi-\xi_1} \Bigr| \le \frac{(  K_1\vee K_2)^{1/2}}{(M_1\wedge M_2)^{1/2}} $.
  Then again by (\ref{lem1proof2}), Lemma~\ref{le1}  and Lemma~\ref{le2}, we have that in the considered case
  $$
  |B^1_{\tau,\xi,q}| \lesssim  
  \frac{(  M_1\wedge  M_2)^{1/2}}{(M_1\vee M_2)^{1/2}} 
 \Bigl\langle
 \frac{(  K_1\vee K_2)^{1/2}(  M_1\wedge  M_2)}{(M_1\vee M_2)^{1/2}}\Bigr\rangle \nonumber \\
 \lesssim   (  K_1\vee K_2)^{1/2}(M_1\wedge M_2)^{1/2}
  $$
  \item[2.2)] $ \; 1\ge \Bigl| \frac{q_1}{\xi_1} -\frac{q-q_1}{\xi-\xi_1} \Bigr| \ge \frac{(  K_1\vee K_2)^{1/2}}{(M_1\wedge M_2)^{1/2}} $.
  Since
  \begin{multline}\label{lem1proof4}
  \Bigl|\frac{\partial}{\partial q_1}
 \Big( \tau-\xi^3+q^2/\xi
 +3 \xi \xi_1 (\xi-\xi_1)
 \\
 +\frac{\xi_1(\xi-\xi_1)}{\xi}
 \Big(\frac{q_1}{\xi_1} -\frac{q-q_1}{\xi-\xi_1}\Big)^2\Big)  \Bigr|
 = 2 \Bigl| \frac{q_1}{\xi_1} -\frac{q-q_1}{\xi-\xi_1} \Bigr|
  \end{multline}
  it follows from Lemma~\ref{le2} that the cardinal of the $ q_1 $ section at fixed $ \xi_1 $ is bounded by
\begin{equation*}
  2 (K_1\vee K_2)  \frac{(  M_1\wedge M_2)^{1/2}}{(  K_1\vee  K_2)^{1/2}}+1
 \lesssim  (  K_1\vee K_2)^{1/2}(  M_1\wedge  M_2)^{1/2} \;.
 \end{equation*}
   According to (\ref{lem1proof2}), the projection of this region on the $ \xi_1 $ axis is of measure less than a uniform constant and thus
   in the considered case
   \begin{equation*}
     |B^1_{\tau,\xi,q}|  \lesssim  (  K_1\vee K_2)^{1/2}(  M_1\wedge  M_2)^{1/2}\; .
   \end{equation*}
  \end{enumerate}
This completes the proof of Lemma~\ref{lem1}.
\end{proof}
We now divide $ B^2_{\tau,\xi,q} $ into three regions by setting
\begin{equation*}
B^{2,1}_{\tau,\xi,q}:=\Big\{(\xi_1,q_1)\in B^2_{\tau,\xi,q}  \,:\,
\Big|\frac{q_1}{\xi_1}-\frac{q-q_1}{\xi-\xi_1}\Big|< \frac{(  K_1\vee K_2)^{1/2}}{(  M_1\wedge  M_2)^{1/2}}  \Big\},
\end{equation*}
\begin{equation*}
B^{2,2}_{\tau,\xi,q}  :=  \Big\{(\xi_1,q_1)\in B^2_{\tau,\xi,q} \,:\, \frac{(  K_1\vee K_2)^{1/2}}{(  M_1\wedge  M_2)^{1/2}} \le
\Big|\frac{q_1}{\xi_1}-\frac{q-q_1}{\xi-\xi_1}\Big| <\frac{(  M_1\vee M_2)^{1/2}}{(  M_1\wedge  M_2)^{1/2}}  \Big\}
\end{equation*}
and
\begin{multline*}
B^{2,3}_{\tau,\xi,q}  :=  \Big\{(\xi_1,q_1)\in B^2_{\tau,\xi,q}\,:\,
\\
 \max\Big( \frac{(  K_1\vee K_2)^{1/2}}{( M_1\wedge
M_2)^{1/2}} , \frac{(  M_1\vee M_2)^{1/2}}{(  M_1\wedge  M_2)^{1/2}}  \Big) \le \Big|\frac{q_1}{\xi_1}-\frac{q-q_1}{\xi-\xi_1}\Big|\,   \Big\}.
\end{multline*}
Note that $ B^{2,1}_{\tau,\xi,q} $ and $ B^{2,2}_{\tau,\xi,q} $
may be empty.
\begin{lem}\label{lem2}
The following estimates hold  whenever  $ M_1\wedge M_2 \ge 1 $
\begin{equation}
|B^{2,1}_{\tau,\xi,q}| \lesssim  (K_1\vee K_2)  \frac{(
M_1\wedge  M_2)^{1/2}}{(  M_1\vee  M_2)^{1/2}} 
\label{est1lem2}
\end{equation}
and
\begin{equation}\label{est2lem2}
|B^{2,2}_{\tau,\xi,q}| \lesssim (K_1\vee K_2)^{1/2} (M_1\wedge
M_2)^{1/2}  \; .
\end{equation}
\end{lem}
\begin{proof}
Assuming that $B^{2,1}_{\tau,\xi,q} $ is not empty,
it follows from \eqref{uio} and Lemma~\ref{le3} that the measure
of its projection on the $ \xi_1 $ axis  can be estimated as
$$
\lesssim\frac{(  K_1\vee K_2)^{1/2}}{(  M_1\vee  M_2)^{1/2}} \; .
$$
On the other hand, it follows from Lemma~\ref{le2} that 
 for a fixed $ \xi_1$
 the cardinal of its  $ q_1$ section is bounded by
 \begin{equation} \label{lem2proof1}
\lesssim (M_1\wedge M_2)^{1/2} (K_1\vee K_2)^{1/2}+1 
\lesssim (M_1\wedge M_2)^{1/2} (K_1\vee K_2)^{1/2}\,.
 \end{equation}
 This proves (\ref{est1lem2}) in view of Lemma~\ref{le1}. 
 Now coming back to (\ref{uio}) and using Lemma~\ref{le3}  we infer that the projection of $ B^{2,2}_{\tau,\xi,q} $ on the $ \xi_1 $ axis
  is bounded by
  $$
 \frac{1}{(M_1\vee M_2)^{1/2}} \Bigl[(K_1\vee K_2)+(M_1\vee M_2) \Bigr]^{1/2}\le C
  $$
  since $ B^{2,2}_{\tau,\xi,q} $ is empty as soon as $ K_1\vee K_2 \ge M_1\vee M_2 $. On the other hand, it follows from  (\ref{lem1proof4}) and Lemma \ref{le1}.
 that  for a fixed $ \xi_1$
 the cardinal of its  $ q_1$ section is also  bounded by (\ref{lem2proof1}). This leads to (\ref{est2lem2}) thanks to Lemma~\ref{le1}.
This completes the proof of Lemma~\ref{lem2}.
\end{proof}
We finally estimate the size of $B^{2,3}_{\tau,\xi,q}$.
\begin{lem}\label{lem3}
For any $ 0<\varepsilon\ll 1 $ it holds
\begin{equation}\label{est1lem3}
|B^{2,3}_{\tau,\xi,q}| \lesssim C_\varepsilon 
\frac{(M_1\wedge M_2)^{1/2+\varepsilon/2}}
{(M_1\vee M_2)^{1/2-\varepsilon/2} }
(K_1\vee K_2) + |I_1|\wedge |I_2| \; .
\end{equation}
\end{lem}
\begin{proof}
We subdivide $ B^{2,3}_{\tau,\xi,q} $ by setting
$$
B^{2,3}_{\tau,\xi,q}(+\infty):=\Big\{(\xi_1, q_1)\in
B^{2,3}_{\tau,\xi,q}\,:\,
\Big|\frac{q_1}{\xi_1}-\frac{q-q_1}{\xi-\xi_1}\Big|\ge |K_1\vee K_2| \Big\}
$$
and
$$
B^{2,3}_{\tau,\xi,q}(L):=\Big\{(\xi_1, q_1)\in B^{2,3}_{\tau,\xi,q}\,:\,
\Big|\frac{q_1}{\xi_1}-\frac{q-q_1}{\xi-\xi_1}\Big|\sim L \Big\}
$$
where $ L=2^l $, $l_0\le l\le l_1 $ with
$$
l_0 \sim \ln
\Bigl(\max (\frac{  M_1\vee M_2}{ M_1\wedge  M_2}, \frac{ K_1\vee
K_2}{ M_1\wedge  M_2})\Bigr),\quad
l_1\sim \ln (K_1\vee K_2)\,.
 $$
In view of (\ref{lem1proof4}), for a fixed $ \xi_1 $, the  $ q_1
$-section of  $ B^{2,3}_{\tau,\xi,q}(+\infty) $ contains at most
two elements and thus by Lemma~\ref{le1},
\begin{equation} \label{lem4est1}
|B^{2,3}_{\tau,\xi,q}(+\infty) | \lesssim |I_1|\wedge |I_2|\; .
\end{equation}
Now, in $ B^{2,3}_{\tau,\xi,q}(L) $ it holds
$$
|\tau-\xi^3+q^2/\xi+3 \xi \xi_1 (\xi-\xi_1)| \lesssim K_1\vee K_1+(M_1\wedge M_2) L^2
$$
and thus from Lemma~\ref{le3} we infer that the measure of the projection of this region on the $ \xi_1 $-axis is bounded by
$$
 \frac{(  K_1\vee K_2)^{1/2}}{(M_1\vee M_2)^{1/2}}+
 \frac{(  M_1\wedge  M_2)^{1/2}}{(M_1\vee M_2)^{1/2}} L \; .
$$
Interpolating with the crude bound $ M_1\wedge M_2 $ we obtain  the  bound
$$
\Bigl[ \frac{(  K_1\vee K_2)^{1/2}}{(M_1\vee M_2)^{1/2}}+
 \frac{(  M_1\wedge  M_2)^{1/2}}{(M_1\vee M_2)^{1/2}} L\Bigr]^{1-\varepsilon} (M_1\wedge M_2)^\varepsilon \; .
$$
On the other hand, for a fixed $ \xi_1 $,  by using again (\ref{lem1proof4}) and Lemma~\ref{le2}, we obtain that
the cardinal of its $ q_1 $-section is bounded by
$$
\frac{K_1\vee K_2}{L} +1 \lesssim \frac{K_1\vee K_2}{L} \; .
$$
We thus get by Lemma~\ref{le1},
\begin{eqnarray*}
| B^{2,3}_{\tau,\xi,q}(L) | & \lesssim & \Bigl[ \frac{(  K_1\vee
K_2)^{1/2}}{(M_1\vee M_2)^{1/2}}+
 \frac{(  M_1\wedge  M_2)^{1/2}}{(M_1\vee M_2)^{1/2}} L\Bigr]^{1-\varepsilon} (M_1\wedge M_2)^\varepsilon
  \frac{(K_1\vee K_2)}{L}   \\
  & \lesssim & \frac{(M_1\wedge M_2)^{1/2+\varepsilon/2} (K_1\vee K_2)}{(M_1\vee M_2)^{1/2-\varepsilon/2} L^\varepsilon},
  \end{eqnarray*}
  where we used that in the considered case $L\geq \frac{(  K_1\vee K_2)^{1/2}}{( M_1\wedge M_2)^{1/2}}$.
  A summation over $L$ yields the claimed bound. This completes the proof of Lemma~\ref{lem3}.
  \end{proof}
 Now, using  Lemmas \ref{lem1}-\ref{lem3}, we get
  $$
  | B_{\tau,\xi,q}| \lesssim
  (K_1\vee K_2)^{1/2} (M_1\wedge M_2)^{1/2}  +
  C_\varepsilon  \frac{(M_1\wedge M_2)^{1/2} (K_1\vee K_2)}{(M_1\vee M_2)^{1/2-\varepsilon} }
+  |I_1|\wedge |I_2| \; .
  $$
  and thus according to (\ref{eq8})-(\ref{eq9}),
     \begin{multline*}
  \| u_1\star u_2\|_{L^2} \lesssim  C_\varepsilon (K_1\wedge K_2)^{1/2}(K_1\vee K_2)^{1/4} 
  \\
 \Bigl[ (M_1\wedge M_2)^{1/4} \left\langle
  \frac{(K_1\vee K_2)^{1/4} }{(M_1\vee M_2)^{1/4-\varepsilon} } \right\rangle
  +\frac{ (|I_1|\wedge |I_2|)^{1/2}}{(K_1\vee K_2)^{1/4} }\Bigr] 
  \| u_1\|_{L^2}   \| u_2\|_{L^2} \; .
\end{multline*}
 This completes the proof of Proposition~\ref{propo1}. 
 \qed
 \begin{coro}\label{coro1}
Let $ u_1 $ and $ u_2 $ be two real valued $ L^2 $ functions
defined on $ \R\times \R\times \Z $ with the following support
properties
\begin{equation*}
(\tau,\xi,q)\in {\rm supp}\, (u_i) \Rightarrow |\xi|\sim M_i ,
\langle \tau-\xi^3+q^2/\xi \rangle \sim K_i , \, i=1,2.
\end{equation*}
Then the following estimates hold, 
\begin{multline} \label{coroest1}
\| u_1\star u_2\|_{L^2(|\xi|\sim M)} \lesssim   (K_1\wedge
K_2)^{1/2}(M_1\wedge M_2\wedge M)^{1/2} 
\\
\langle(K_1\vee K_2)^{1/4} (M_1\wedge M_2)^{1/4}\rangle \| u_1\|_{L^2}\| u_2\|_{L^2}
\end{multline}
and if $ M_1\wedge M_2 \ge 1 $  then
\begin{multline}\label{coroest2}
  \| u_1\star u_2\|_{L^2(|\xi|\sim M)} \lesssim (K_1\wedge K_2)^{1/2}(K_1\vee K_2)^{1/4}
 \\
  \Bigl[ (M_1\wedge M_2)^{1/4} \left\langle
  \frac{(K_1\vee K_2) }{(M_1\vee M_2)^{1-} }\right\rangle^{1/4} 
  +\frac{ (M_1\wedge M_2\wedge M)^{1/2}}{(K_1\vee K_2)^{1/4}}\Bigr]\| u_1\|_{L^2}   \| u_2\|_{L^2} \, .
\end{multline}
 \end{coro}
 \begin{proof}
Rewriting the functions $ u_i $ as $ u_i=\sum_{k\in \Z} u_{i,k} $ with 
$$
u_{i,k}= u_i \, \chi_{\{ kM\le \xi <(k+1) M\}}\; ,
$$
we easily obtain by support considerations and Minkowsky inequality that
\begin{eqnarray*}
\| u_1\star u_2\|_{L^2(|\xi|\sim M)}  &= & \Bigr\|\sum_{k\in \Z} \sum_{0\le |q|\lesssim 1 }u_{1,k} \star u_{2,-k+q} \Bigl\|_{L^2(|\xi|\sim M)} \\
 & \lesssim &  \sum_{0\le q\lesssim 1 }\sum_{k\in \Z} \| u_{1,k}\star u_{2,-k+q}\|_{L^2} 
\end{eqnarray*}
The desired result follows by applying Proposition \ref{propo1} with $ |I_1|=|I_2|= M $ for each $ k\in \Z $, and then  Cauchy-Schwarz in $ k$.
\end{proof}
A rough interpolation between (\ref{pro1est1}) and (\ref{pro1est2}) on one side and between
 (\ref{coroest1}) and (\ref{coroest2}) on the other side we get the following bilinear estimates that we will use intensively in the next section.
\begin{coro}\label{coro2}
There exists $\beta_0<1/2$ such that the following holds true. Let $ u_1 $ and $ u_2 $ be two real valued $ L^2 $ functions
defined on $ \R\times \R\times \Z $ with the following support properties
\begin{equation*}
(\tau,\xi,q)\in {\rm supp}\, (u_i) \Rightarrow |\xi|\sim M_i\ge 1 ,
\langle \tau-\xi^3+q^2/\xi \rangle \sim K_i , \, i=1,2.
\end{equation*}
Then the following estimates hold ,
\begin{equation} \label{coro2est1}
\| u_1\star u_2\|_{L^2} \lesssim  (K_1 K_2)^{\beta_0} (M_1\wedge M_2)^{1/2} \| u_1\|_{L^2}  \|  u_2\|_{L^2}
\end{equation}
and
\begin{multline}\label{coro2est2}
  \| u_1\star u_2\|_{L^2(|\xi|\sim M)}  \lesssim   (K_1 K_2)^{\beta_0}\big[(M_1\wedge M_2)^{1/4}+ (M_1\wedge M_2\wedge M)^{1/2}\big] 
  \\
  \| u_1\|_{L^2}  \|  u_2\|_{L^2} \, .
\end{multline}
 \end{coro}
 \begin{proof}
 Let us first prove \eqref{coro2est1}. If $M_1\wedge M_2\leq 1$ then we can easily conclude by only using (\ref{pro1est1}).
 If $M_1\wedge M_2\geq 1$ then (\ref{pro1est2}) gives
 \begin{multline*}
  \| u_1\star u_2\|_{L^2} \lesssim (K_1\wedge K_2)^{1/2}(K_1\vee K_2)^{1/4}
   (M_1\wedge M_2)^{1/4}
   \\
   \left(1+
  \frac{(K_1\vee K_2)^{1/4} }{(M_1\vee M_2)^{1/4-} }
  +\frac{ (M_1\wedge M_2)^{1/4}}{(K_1\vee K_2)^{1/4} }\right)
  \| u_1\|_{L^2}   \| u_2\|_{L^2} \, .
\end{multline*}
We now distinguish cases according to which terms dominates in the sum
$$
1+  \frac{(K_1\vee K_2)^{1/4} }{(M_1\vee M_2)^{1/4-} } +\frac{ (M_1\wedge M_2)^{1/4}}{(K_1\vee K_2)^{1/4} }\,.
$$
If the first or the third term dominates then we get directly the needed bound. The only case we really need an interpolation between
(\ref{pro1est1}) and (\ref{pro1est2})  is when the second term dominate. In this case we interpolate with weight $\theta$ on the bound
(\ref{pro1est1}) and weight $1-\theta$ on the second bound (\ref{pro1est2}) and the conditions on $\theta$ to get the needed estimate is
$$
\theta>0,\quad \frac{3}{4}\theta<\frac{1}{2}
$$ 
which of course can be achieved. Let us now turn to the proof of \eqref{coro2est2}. 
Again,  if $M_1\wedge M_2\leq 1$ then we can easily conclude by only using (\ref{coroest1}).
If $M_1\wedge M_2\geq 1$ then (\ref{coroest2}) gives
\begin{multline*}
  \| u_1\star u_2\|_{L^2(|\xi|\sim M)} \lesssim (K_1\wedge K_2)^{1/2}(K_1\vee K_2)^{1/4}
 \\
  \Big[ (M_1\wedge M_2)^{1/4}  \Big(1+
  \frac{(K_1\vee K_2)^{1/4} }{(M_1\vee M_2)^{1/4-} }
  \Big)
  +\frac{ (M_1\wedge M_2\wedge M)^{1/2}}{(K_1\vee K_2)^{1/4}}\Big]\| u_1\|_{L^2}   \| u_2\|_{L^2} \, .
\end{multline*}
Again, we distinguish cases according to which terms dominates in the sum
$$
(M_1\wedge M_2)^{1/4} +  \frac{(M_1\wedge M_2)^{1/4} (K_1\vee K_2)^{1/4} }{(M_1\vee M_2)^{1/4-} }
  +\frac{ (M_1\wedge M_2\wedge M)^{1/2}}{(K_1\vee K_2)^{1/4}}\,.
  $$
If the first or the third term dominates then we get directly the needed bound. 
If the second term dominates than 
 we interpolate with weight $\theta$ on the bound
(\ref{coroest1}) and weight $1-\theta$ on the bound 
(\ref{coroest2}) and the conditions on $\theta$ to get the needed estimate is
$$
\theta>0,\quad \theta<\frac{1}{2},
$$ 
the first assumption being imposed in order to ensure the $K_1, K_2$ conditions and the second one for the $M$, $M_1$ and $M_2$ conditions.
 This completes the proof of Corollary~\ref{coro2}.
 \end{proof}
\section{Global well-posedness on $ \R\times \T $}
\subsection{Two bilinear estimates in $ X^{b,b_1,s}
$-spaces}\label{31}
The local well-posedness result is a direct consequence of the following  bilinear estimates combined with a standard fixed
 point argument
 in $ X^{b,b_1,s} $ spaces
\begin{pro}\label{probilest_bis}
There exist $\beta<1/2$  and $ 1/4<b_1<3/8 $  such that for all $u,v \in  X^{1/2,b_1,s}$, the following bilinear estimates hold
 \begin{equation} \label{bilinearestimate1}
 \|\partial_x(uv) \|_{X^{-1/2,b_1,s}} \lesssim
  \|u\|_{X^{1/2,b_1,s}} \|v\|_{X^{\beta,0,s}} +  \|u\|_{X^{\beta,0,s}} \|v\|_{X^{1/2,b_1,s}} 
 \end{equation}
 and
  \begin{equation} \label{bilinearestimate2}
 \|\partial_x(uv) \|_{ Z^{-1/2,s}} \lesssim 
 \|u\|_{X^{1/2,b_1,s}} \|v\|_{X^{\beta,0,s}}  + \|u\|_{X^{\beta,0,s}} \|v\|_{X^{1/2,b_1,s}} 
 \; ,
 \end{equation}
 provided $ s\geq 0 $.
\end{pro}
To prove this bilinear estimate we first note  that by symmetry it
suffices to consider $ \partial_x \Lambda(u,v) $ where
$\Lambda(\cdot,\cdot) $ is defined by
$$
{\mathcal F}_x (\Lambda(u,v)):=\int_{\R} \chi_{|\xi_1|\le
|\xi-\xi_1|} ({\mathcal F}_x u) (\xi_1) ({\mathcal F}_x v)
(\xi-\xi_1)\, d\xi_1\; .
$$
The following resonance relation (see \cite{Bo3}) is crucial for our analysis :
\begin{equation}\label{resonant}
|\sigma-\sigma_1-\sigma_2|=\Bigl|3\xi \xi_1 (\xi-\xi_1)+\frac{(\xi q_1-\xi_1 q)^2}{\xi \xi_1 (\xi-\xi_1)}\Bigr|\ge 3 |\xi \xi_1 (\xi-\xi_1)|,
\end{equation}
where 
$$
\sigma:=\sigma(\tau,\xi,q):=\tau-\xi^3+q^2/\xi,\quad
 \sigma_1:=\sigma(\tau_1,\xi_1,q_1),
 $$
and
$\sigma_2:=\sigma(\tau-\tau_1,\xi-\xi_1,q-q_1)$.
We start by  giving an estimate for interactions of high
frequencies in $ x$ whose sum is not close to zero.
\begin{lem}\label{lemme1}
There exists $\beta<1/2$ such that
$$
\|\partial_x P_{\ge 1}\Lambda(P_{\ge 1} u,v) \|_{X^{-1/2,b_1,s} } 
\lesssim 
 \|u\|_{X^{1/2,b_1,s}} \|v\|_{X^{\beta,0,s}} +  \|u\|_{X^{\beta,0,s}} \|v\|_{X^{1/2,b_1,s}} \; ,
 $$
 provided $ s\ge 0 $ and $ 1/4<b_1<3/8 $.
 \end{lem}
 \begin{proof}
 By duality we have to prove that
 \begin{multline}\label{estbil1}
I:=\Bigl| \int_{\R^4} \sum_{(q,q_1)\in \Z^2}
\Gamma^{\tau,\xi,q}_{\tau_1,\xi_1,q_1} f(\tau_1,\xi_1,q_1)
g(\tau_2,\xi_2,q_2) h(\tau,\xi,q) \, d\tau \, d\tau_1 \, d\xi \,
d\xi_1  \Bigr|
 \\
\lesssim 
(
\|f\|_{X^{1/2,b_1,s}} \|g\|_{X^{\beta,0,s} }
+
\|f\|_{X^{\beta,0,s}} \|g\|_{X^{1/2,b_1,s} }
)
\|h\|_{2}
 \end{multline}
where
\begin{equation*} 
 \Gamma^{\tau,\xi,q}_{\tau_1,\xi_1,q_1}:= 
\chi_{\{1\le |\xi_1|\le |\xi-\xi_1|, \, |\xi|\ge 1\}}
\langle \sigma \rangle^{-1/2}
\Bigl\langle\frac{\langle \sigma \rangle}
{\langle \xi \rangle^3}\Bigr\rangle^{b_1} |\xi| \langle \zeta \rangle^{s} \langle \zeta_1 \rangle^{-s} \langle \zeta-\zeta_1\rangle^{-s} \,.
\end{equation*}
For $s\geq 0$,
\begin{equation}\label{tame}
\langle \zeta \rangle^{s}\lesssim \langle \zeta_1 \rangle^{s} +\langle \zeta-\zeta_1 \rangle^{s}\,.
\end{equation}
We therefore need to prove \eqref{estbil1} only for $s=0$. 
In addition, we obtain that thanks to \eqref{tame} the estimate \eqref{estbil1} holds with the following left hand-side
\begin{multline*}
(
\|f\|_{X^{1/2,b_1,s}} \|g\|_{X^{\beta,0,0} }
+
\|f\|_{X^{1/2,b_1,0}} \|g\|_{X^{\beta,0,s} }
\\
+
\|f\|_{X^{\beta,0,s}} \|g\|_{X^{1/2,b_1,0} }
+
\|f\|_{X^{\beta,0,0}} \|g\|_{X^{1/2,b_1,s} }
)
\|h\|_{2}\,.
\end{multline*}
The above refinement allows us to get tame estimates which provides the propagation of regularity in the proof of Theorem~\ref{theo1}.
We will not further detail this (standard) aspect of the analysis in the sequel.

 We can of course assume that $ f$, $ g$  and $ h$ are
 non-negative functions in $ \R^2\times \Z $.
 For any $ L^2 $ function $w $ of $
(\tau,\xi,q) $  and any $ (K,M)\in [1,+\infty[\times \R^*_+$ we
define the following localized versions of $ w$ :
\begin{equation}
w_{K}:=w\chi_{\{\langle \sigma\rangle \sim K\}}, \;  w_{M}:=w\chi_{\{|\xi|\sim
 M\}} \mbox{ and }
 w_{K,M}:=w \chi_{\{\langle \sigma\rangle \sim K,\, |\xi|\sim
 M\}}\; .
 \end{equation}
We now consider the
 dyadic level
 \begin{eqnarray*}
 D^{K,K_1,K_2}_{M,M_1,M_2} &:= & \Bigl\{(\tau,\xi,q,\tau_1,\xi_1,q_1):
 |\xi| \sim M, \, |\xi_1|\sim M_1  \\
  & & |\xi-\xi_1|\sim M_2, \,
 \langle \sigma \rangle\sim K,   \langle \sigma_1 \rangle\sim K_1,
   \langle \sigma_2 \rangle\sim K_2\Bigr\} \; .
 \end{eqnarray*}
 Denoting by $ J^{K,K_1,K_2}_{M,M_1,M_2} $ the contribution of
 $ D^{K,K_1,K_2}_{M,M_1,M_2}$ to (\ref{estbil1}), then clearly
 \begin{equation}\label{nantes}
 I\lesssim \sum_{K,K_1,K_2, M,M_1\le M_2}J^{K,K_1,K_2}_{M,M_1,M_2}
\end{equation}
where $ K,K_1,K_2, M,M_1 $ and $ M_2$ describe the dyadic level $2^{\N}$. 
From the resonance estimate 
(\ref{resonant}) it follows that $
D^{K,K_1,K_2}_{M,M_1,M_2}$ is empty whenever $ \max(K,K_1, K_2) \lesssim M M_1 M_2 $. We  thus have only to consider the three
following contributions:  
\\
$\bullet $ {\bf  A.} $ K \gtrsim M M_1 M_2 $\\
$\bullet $ {\bf  B.} $ K_1 \gtrsim M M_1 M_2 $ and  $ K \lesssim  M M_1 M_2 $ \\
$\bullet $ {\bf  C.}  $ K_2 \gtrsim  M M_1 M_2  $ and $ \max(K,K_1) \lesssim M M_1 M_2$
\\
Moreover, since $ M_1\le M_2 $, it is clear that either 
\\
$\bullet$  {\bf 1.}
$ M\lesssim M_1\sim M_2 $ 
\\
or 
\\
$\bullet${ \bf 2.} $  M_1\lesssim M\sim M_2 $.
\\
We will now estimate all these contributions. 
\\
 \begin{enumerate}
 \item[{\bf  A.1.}]  $ K \gtrsim  M M_1 M_2 $ and $ M\lesssim M_1\sim M_2$. In this case we can write $ K\sim 2^l M M_2^2 $ with $ l\in \Z_+ $.
Moreover one has clearly $ K\gtrsim  M^3 $. From (\ref{coro2est2}) we thus have
 \begin{multline*}
I  \lesssim  
\sum_{K,K_1,K_2,M_1,M\lesssim M_2} 
MK^{-1/2}(1+K^{b_1}/M^{3b_1})
\langle f_{M_1 K_1}\star g_{M_2,K_2}, h_{M,K}\rangle
\\
\lesssim
\sum_{K_1,K_2,M\lesssim M_2, l\in \Z_+} (K_1 K_2)^{\beta_0}
 \frac{(M_2^{1/4}+M^{1/2}) M^{1-3b_1}}{2^{(1/2-b_1)l} M^{1/2-b_1} M_2^{1-2b_1}  } 
  \\
  \| f_{M_2,K_1}\|_{2} \| g_{M_2,K_2}\|_{2}   \| h_{M,2^l M M_2^2}\|_{2} 
  \\
   \lesssim \sum_{K_1,K_2,M\lesssim M_2, l\in \Z_+} 2^{-(1/2-b_1)l}( K_1 K_2)^{\beta_0} \Bigl[M_2^{-1/4}+\Bigl( \frac{M}{M_2}\Bigr)^{1-2b_1}
  \Bigr]
  \\
   \| f_{M_2,K_1}\|_{2}
   \| g_{M_2,K_2}\|_{2}   \| h_{M,2^l M M_2^2}\|_{2} 
 \end{multline*}
Here and in the sequel we use a slight abuse of notation by denoting still by $I$ the contribution of the region of dyadic values under consideration to $I$ (see \eqref{nantes}).
 Summing over $ K_1 $, $ K_2 $, $l$ and over $ M$ taking values $\lesssim M_2$, we get for $\beta\in (\beta_0,1/2)$,
 $$
 I\lesssim  \sum_{M_2}\| f_{M_2}\|_{X^{\beta,0,0}}  \| g_{M_2}\|_{X^{\beta,0,0}}   \| h \|_{2} \\
$$
and Cauchy-Schwarz in $ M_2 $ yields
\begin{equation*}
 I\lesssim  (\sum_{M_2} \| f_{M_2}\|_{X^{\beta,0,0}}^2)^{1/2} (\sum_{M_2} \| g_{M_2}\|_{X^{\beta,0,0}}^2)^{1/2}\| h \|_{2} 
 \lesssim  \| f\|_{X^{\beta,0,0}}\| g\|_{X^{\beta,0,0}}   \| h \|_{2} \; .
   \end{equation*}
\item[{\bf  A.2.}]  $ K \gtrsim M M_1 M_2 $ and $ M_1\lesssim M\sim M_2$. In this case
we can write  $ K\sim 2^l M_1 M_2^2 $ with $ l\in \Z_+ $. We distinguish between the cases
 $ K<M^3$ and $ K\ge M^3 $. 
 \\
 i) $ K<M^3$. Then according to (\ref{coro2est1}), it holds
  \begin{multline*}
 I  \lesssim
 \sum_{K_1,K_2\atop M_1\le M_2, l\geq 0}\hspace*{-2mm}(K_1 K_2)^{\beta_0}
  \frac{M_1^{1/2}M_2}{2^{l/2} M_1^{1/2} M_2 } \| f_{M_1,K_1}\|_{2}
   \| g_{M_2,K_2}\|_{2}   \| h_{M_2,2^l M_1 M_2^2}\|_{2} 
   \\
 \lesssim  \sum_{K_1,K_2\atop M_1\le M_2, l\geq 0}
  \hspace*{-2mm}2^{-l/2} (K_1 K_2)^{\beta_0}  
  \| f_{M_1,K_1}\|_{2}   \| g_{M_2,K_2}\|_{2}   \| h_{M_2,2^l M_1 M_2^2}\|_{2} \; .
 \end{multline*}
Summing over $ K_1$, $ K_2 $ and applying Cauchy-Schwarz in $ (M_1,M_2) $ it leads to
$$
I \lesssim   \sum_{l\geq 0}2^{-l/2}(\sum_{M_1}\|f_{M_1}\|_{X^{\beta,0,0}}^2)^{1/2}(\sum_{M_2}\|g_{M_2}\|_{X^{\beta,0,0}}^2)^{1/2}
(\sum_{M_1,M_2}\|h_{M_2,2^lM_1M_2^2}\|_2^2)^{1/2},
$$
provided $\beta\in(\beta_0,1/2)$ which leads to the needed bound.
\\
ii) $K\ge M^3 $. Then using (\ref{coro2est2})  we have
\begin{eqnarray*}
 I & \lesssim &\hspace*{-4mm} \sum_{K_1,K_2\atop M_1\le M_2, l\geq 0}(K_1 K_2)^{\beta_0}
  \frac{M_1^{1/2} M_2^{1-3b_1}}{2^{(1/2-b_1)l} M_1^{1/2-b_1} M_2^{1-2b_1}   } \\
   & & \hspace*{20mm}\| f_{M_2,K_1}\|_{2}
   \| g_{M_2,K_2}\|_{2}   \| h_{M,2^l M M_2^2}\|_{2} \\
  & \lesssim &\hspace*{-4mm} \sum_{K_1,K_2\atop M_1\le M_2, l\geq 0} 2^{-(1/2-b_1)l} (K_1 K_2)^{\beta_0}
  \Bigl( \frac{M_1}{M_2}\Bigr)^{b_1}
   \| f_{M_1,K_1}\|_{2}
   \| g_{M_2,K_2}\|_{2}   \| h_{M,2^l M_1 M_2^2}\|_{2}
 \end{eqnarray*}
 Summing as in  A.1. exchanging the role of $ M $ and $ M_1 $ and using that $ b_1>0$ yields the needed bound.
 \item[{\bf  B.1.}]  $ K_1 \gtrsim M M_1 M_2 $, $ K\lesssim M M_1 M_2 $  and $ M\lesssim M_1\sim M_2 $. \\
 Then  we can write $ K_1\sim 2^l M M_2^2 $ with $ l\in \N $. We separate again the cases
  $ K< M^3 $ and $  K\ge M^3 $ and use that $K_1\gtrsim K $. \\
  i) $ K<M^3$. Then we get using Corollary~\ref{coro2}
  \begin{multline*}
 I\lesssim 
 \sum_{K,K_1,K_2,M\lesssim M_1, M_2 } MK^{-1/2}(1+K^{b_1}/M^{3b_1})
\langle f_{M_1 K_1}, \check{g}_{M_2,K_2}\star h_{M,K}\rangle
 \\
  \lesssim  \sum_{K,K_2\atop M\le M_2, l\geq 0} (K K_2)^{\beta_0}
   \frac{M^{1/2} M (2^{l/2} M^{1/2} M_2)}{(2^{l/2} M^{1/2} M_2) K^{1/2}
    } \| f_{M_2,2^l M M_2^2}\|_{2}
   \| g_{M_2,K_2}\|_{2}   \| h_{M,K}\|_{2} 
   \\
   \lesssim  \sum_{K_,K_2 \atop M\le M_2, l\geq 0}2^{-l/2} K^{-(1/2-\beta_0)} \Bigl( \frac{M}{M_2}\Bigr)
\\
 (2^{l/2} M^{1/2} M_2) \| f_{M_2,2^l M M_2^2}\|_{2}
 K_2^{\beta_0}  \| g_{M_2,K_2}\|_{2}   \| h_{M,K}\|_{2}\; .
 \end{multline*}
 Summing as in A.1.  yields for $\beta\in (\beta_0,1/2)$, 
$$
I\lesssim
 \| f\|_{X^{1/2,0,0}}\| g\|_{X^{\beta,0,0}}   \| h \|_{2}\,.
$$
 Observe that in this case and in some places of the sequel we use that the localization assumptions in the bilinear
 estimates established in the previous section are invariant under the transformation 
$(\tau,\xi,q)\rightarrow (-\tau,-\xi,-q)$.
 \\
ii) $K\ge M^3 $. Then we have using Corollary~\ref{coro2}
\begin{multline*}
 I  \lesssim  \sum_{K,K_2\atop M\le M_2, l\geq 0}K^{\beta_0}   \frac{M^{1/2} M^{1-3b_1}}{2^{(1/2-b_1)l} M^{1/2-b_1} M_2^{1-2b_1} K^{1/2} } 
\\
   (2^{l/2} M^{1/2} M_2)
\| f_{M_2,2^l M M_2^2}\|_{2}
  K_2^{\beta_0} \| g_{M_2,K_2}\|_{2}   \| h_{M,K}\|_{2} 
\\
 \lesssim  \sum_{K,K_2\atop M\le M_2, l\geq 0} 2^{-(1/2-b_1)l} K^{-(1/2-\beta_0)} \Bigl( \frac{M}{M_2}\Bigr)^{1-2b_1}
\\   
(2^{l/2} M^{1/2} M_2)
\| f_{M_2,2^l M M_2^2}\|_{2}
  K_2^{\beta_0} \| g_{M_2,K_2}\|_{2}   \| h_{M,K}\|_{2} \,.
 \end{multline*}
 This implies the result as above.
\item[{\bf  B.2.}] $ K_1 \gtrsim M M_1 M_2 $, $ K\lesssim M M_1 M_2 $  and $M_1\lesssim M\sim M_2 $.  Then
  one can write $ K_1\sim 2^l M_1 M_2^2 $ with $ l\in \N $ and clearly  $ K\lesssim M^3\sim M^3_2 $.  
 Using Corollary~\ref{coro2},  we obtain
 \begin{multline*}
 I  \lesssim  \sum_{K,K_2 \atop M_1\le M_2, l\geq 0}
K^{-(1/2-\beta_0)}    
 \frac{(M_1^{1/2}+M_2^{1/4})M_2}{2^{l/2} M_1^{1/2} M_2 }
\\   
(2^{l/2} M_1^{1/2} M_2)
\| f_{M_1,2^l M_1 M_2^2}\|_{2}
  K_2^{\beta_0} \| g_{M_2,K_2}\|_{2}   \| h_{M,K}\|_{2} 
   \end{multline*}
   Next, we can write
   \begin{multline*}
   \frac{(M_1^{1/2}+M_2^{1/4})M_2}
{2^{l/2} M_1^{1/2} M_2 }
  \lesssim
  2^{-l/2}   \Big(
   \frac{M_1^{2b_1}}{M_2^{2b_1}}+\frac{1}{M_1^{1/2-2b_1}M_{2}^{2b_1-1/4}}
   \Big)\Big\langle \frac{K_1}{M_1^3}\Big\rangle^{b_1}
   \\
   \lesssim
   2^{-l/2}   \Big(
   \frac{M_1^{2b_1}}{M_2^{2b_1}}+\frac{M_1^{2b_1-1/4}}{M_{2}^{2b_1-1/4}}
   \Big)\Big\langle \frac{K_1}{M_1^3}\Big\rangle^{b_1}
   \end{multline*} 
   and we can conclude as in the previously considered cases thanks to our assumptions on $b_1$.
 Note that this is a case where we need to introduce the additional factor in the Fourier transform restriction norm.
 \item[{\bf  C.1.}] $ K_2 \gtrsim M M_1 M_2 $, $ \max(K,K_1)\lesssim M M_1 M_2 $  and $M\lesssim  M_1\sim M_2 $.
   Then we can write $ K_2\sim 2^l M M_2^2 $ with $ l\in \Z_+ $. Since $ M_1\sim M_2 $ this contribution
   can be treated  exactly
  as  the contribution of the case $ B.1. $ by exchanging the roles of $ K_1 $ and $ K_2 $.
   \item[{\bf C.2.}] $ K_2 \gtrsim M M_1 M_2 $, $ \max(K,K_1) \lesssim M M_1 M_2 $  and $M_1\lesssim M\sim M_2 $.
    Then we can write $ K_2\sim 2^l M_1 M_2^2 $ with $ l\in \N $. In the considered case $K\lesssim M^3$ and thus using
  Corollary~\ref{coro2}, we get  
  \begin{multline*}
 I \lesssim  \sum_{K,K_1 \atop M_1\le M_2, l\in \N}(K  K_1)^{\beta_0}
 \frac{M_1^{1/2} M_2}{(2^{l/2} M_1^{1/2} M_2) 
   K^{1/2}}
\\  
 \| f_{M_1,K_1}\|_{2}
 (2^{l/2} M_1^{1/2} M_2) \| g_{M_2,2^l M_1 M_2^2}\|_{2}   \| h_{M_2,K}\|_{2}   
\\
   \lesssim  \sum_{K_,K_1 \atop M_1\le M_2, l\in \N} 2^{-l/2} K^{-(1/2-\beta_0)} 
\\
K_1^{\beta_0} \| f_{M_1,K_1}\|_{2}
   (2^{l/2} M_1^{1/2} M_2) \| g_{M_2,2^l M_1 M_2^2}\|_{2}   \| h_{M_2,K}\|_{2} \; .
 \end{multline*}
 Summing as in A.2.case $ i) $ yields the result.
\end{enumerate}
This completes the proof of Lemma~\ref{lemme1}.
\end{proof}
Let us now treat interactions of frequencies in $ x$  whose sum is closed to zero.
\begin{lem} \label{lemme2}
There exists $\beta<1/2$ such that
\begin{equation*}
\|\partial_x P_{\le 2} \Lambda(u,v) \|_{X^{-1/2,b_1,s}} 
\lesssim \| u \|_{X^{1/2,b_1,s}} \| v \|_{X^{\beta,0,s}}
+
\| v \|_{X^{1/2,b_1,s}} \| u \|_{X^{\beta,0,s}}
\; ,
 \end{equation*}
provided $s\geq 0$ and $1/4<b_1<3/8$. 
\end{lem}
 \begin{proof}
 By duality it is equivalent prove (\ref{estbil1}) with
\begin{equation*} 
 \Gamma^{\tau,\xi,q}_{\tau_1,\xi_1,q_1}:= \chi_{\{|\xi_1|\le |\xi-\xi_1|, \, |\xi|\le 1\}}
 \langle \sigma \rangle^{-1/2}
\Bigl\langle\frac{\langle \sigma \rangle}
{\langle \xi \rangle^3}\Bigr\rangle^{b_1} |\xi| \langle \zeta \rangle^{s} \langle \zeta_1 \rangle^{-s} \langle \zeta-\zeta_1\rangle^{-s} \,.
\end{equation*}
Again we can restrict our attention to the case $s=0$. We proceed in a similar way as in the proof of Lemma~\ref{lemme1}. 
 The only difference is that here
  $ M$ describes the dyadic levels $ 2^{-\N} $ and $ M_1, \, M_2 $ describe the dyadic levels $ 2^{\Z} $.
  We distinguish between the regions {\bf 1.} $ |\xi_1|\sim |\xi_2|$ and the region  {\bf 2.} $   |\xi|\sim |\xi_2|$.  
  \\
 {\bf 1.}   $|\xi_1|\sim |\xi_2|$. In this region $ |\xi|\lesssim |\xi_1|$. We treat only the case
   $ |\sigma_1|\lesssim |\sigma_2| $ since, as $ |\xi_1|\sim |\xi_2| $, the case $ |\sigma_2|\lesssim |\sigma_1| $ is similar.
    We subdivise this region in the two subregions  $ |\sigma_1|\ge |\sigma |$
  and $
   |\sigma_1|<|\sigma| $. In the first one, according to (\ref{pro1est1}) we infer that
   \begin{eqnarray*}
I   &  \lesssim  & \sum_{K,K_1,K_2\atop M\le 1,M_2}
\frac{(K\wedge K_2)^{1/2} (K\vee K_2)^{1/4} M^{1/2}}{K^{1/2-b_1}}
 \| f _{K_1,M_2}\|_{2}
   \| g _{K_2,M_2}\|_{2}   \| h _{K,M}\|_{2} \\
   &  \lesssim  & \sum_{K,K_1,K_2 \atop M\le 1,M_2} (K_1 K_2)^{3/8} K^{b_1-1/2}
 M^{1/2}  \| f _{K_1,M_2}\|_{2}
   \| g _{K_2,M_2}\|_{2}   \| h _{K,M}\|_{2}
  \end{eqnarray*}
 Summing in $ M$, $K_2$, $K$, $K_1$, using that $ b_1<1/2 $  and applying Cauchy-Schwarz in $ M_2 $ implies the
  desired result. \\
  Now in the subregion $ |\sigma_1|<|\sigma| $ we know from the resonance relation
   (\ref{resonant}) that $ |\sigma|\gtrsim
   |\xi_1||\xi_2||\xi| $ and thus according to (\ref{coroest1})
 \begin{eqnarray*}
I   &  \lesssim  & \sum_{K,K_1,K_2 \atop M\le 1,M_2> 1}
\frac{(K_1\wedge K_2)^{1/2} (K_1\vee K_2)^{1/4} M_2^{1/4}M}{ (M M_2^2)^{1/8} K^{3/8-b_1} }
 \| f _{K_1,M_2}\|_{2}
   \| g _{K_2,M_2}\|_{2}   \| h _{K,M}\|_{2} \\
   & & + \sum_{K,K_1,K_2 \atop M\le 1,0<M_2\le 1} \frac{(K_1\wedge
K_2)^{1/2} (K_1\vee K_2)^{1/4}  M}{ K^{1/2-b_1} }
 \| f _{K_1,M_2}\|_{2}
   \| g _{K_2,M_2}\|_{2}   \| h _{K,M}\|_{2} \\
   &  \lesssim  & \sum_{K,K_1,K_2 \atop M\le 1,M_2} (K_1 K_2)^{3/8} K^{b_1-3/8}
 M^{7/8}  \| f _{K_1,M_2}\|_{2}
   \| g _{K_2,M_2}\|_{2}   \| h _{K,M}\|_{2}
  \end{eqnarray*}
which can be sum in the same way as above for $ b_1<3/8$. \vspace*{2mm}\\
   {\bf 2.}   $ |\xi|\sim |\xi_2|$.  In this region $|\xi_1|\lesssim |\xi|$ and
   $ M, M_1 $ and $ M_2 $  describe only the dyadic levels $ 2^{-\N} $ .
    Since $M_1\lesssim  M\sim M_2 $, it follows from (\ref{pro1est1}) that
       \begin{eqnarray*}
I   &  \lesssim  & \sum_{K,K_1,K_2 \atop M_1\le 1,M_2\le 1}
\frac{(K_1\wedge K_2)^{1/2} (K_1\vee K_2)^{1/4}
M_1^{1/2}}{K^{1/2-b_1}}
 \| f _{K_1,M_1}\|_{2}
   \| g _{K_2,M_2}\|_{2}   \| h _{K,M_2}\|_{2} \\
   &  \lesssim  & \sum_{K,K_1,K_2 \atop M_1\le 1,M_2\le 1} K_1^{3/8} K_2^{3/8} K^{b_1-1/2}
 M_1^{1/2}  \| f _{K_1,M_1}\|_{2}
   \| g _{K_2,M_2}\|_{2}   \| h _{K,M_2}\|_{2}
  \end{eqnarray*}
  Summing in $ M_1$, $K_2$, $K$, $K_1$, using that $ b_1<1/2 $  and applying Cauchy-Schwarz in $ M_2 $ we obtain the
  desired result.   
  This completes the proof of Lemma~\ref{lemme2}.
  \end{proof}
  It remains to treat  interactions between very low frequencies
   and high frequencies. For this purpose we define the operators $ \Lambda_i $, $i=0,1,2 $ by
\begin{eqnarray*}
{\mathcal F}_{t,x,y} (\Lambda_i(u,v))(\tau,\xi,q)& := &
\int_{\R^2}\sum_{q\in\Z} \chi_{\{(\tau,\tau_1,\xi,\xi_1,q,q_1)\in
A_i \}}
 ({\mathcal F}_{t,x,y} u) (\tau_1,\xi_1,q_1) \\
 & & ({\mathcal
F}_{t,x,y} v)(\tau-\tau_1,\xi-\xi_1,q-q_1)\, d\xi_1\, d\tau_1\; .
\end{eqnarray*}
where $$ A_0:=\{(\tau,\tau_1,\xi,\xi_1,q,q_1)\in B,
\,|\sigma|\ge |\xi||\xi_1| |\xi-\xi_1|\}
\; ,$$
 $$
A_1:=\{(\tau,\tau_1,\xi,\xi_1,q,q_1)\in B \,:\,|\sigma_1|\ge |\xi||\xi_1| |\xi-\xi_1|>\max(|\sigma|,|\sigma_2|\} \; ,
$$
$$
 A_2:=\{(\tau,\tau_1,\xi,\xi_1,q,q_1)\in B \,:\, |\sigma_2|\ge |\xi||\xi_1| |\xi-\xi_1|>|\sigma|\}\; , 
 $$
with
$$
B:=\{(\tau,\tau_1,\xi,\xi_1,q,q_1)\in \R^4\times\Z^2 \,:\, |\xi_1|\le |\xi-\xi_1|\} \, . 
$$
\begin{lem}\label{lemme3}
There exists $\beta<1/2$ such that
\begin{multline}\label{estLambda0}
\|\partial_x \Lambda_0(P_{\le 1} u,P_{\ge 2}v)\|_{X^{-1/2,b_1,s}} 
\lesssim 
\\
\| u \|_{X^{1/2,b_1,s}} \| v \|_{X^{\beta,0,s}}+\| u \|_{X^{\beta,0,s}} \| v \|_{X^{1/2,b_1,s}}\,,
 \end{multline}
\begin{multline} \label{estLambda2}
\|\partial_x \Lambda_2(P_{\le 1} u,P_{\ge 2}v)
\|_{X^{-1/2,b_1,s}} \lesssim 
\\
\| u \|_{X^{1/2,b_1,s}} \| v \|_{X^{\beta,0,s}}+\| u \|_{X^{\beta,0,s}} \| v \|_{X^{1/2,b_1,s}}
 \end{multline}
 and
 \begin{multline}\label{R11}
\| \partial_x\Lambda_1 (P_{\le 1} u,P_{\ge 2}v)\|_{X^{-1/2,b_1,s}} 
\lesssim 
\\
\| u \|_{X^{1/2,b_1,s}} \| v \|_{X^{\beta,0,s}}+\| u \|_{X^{\beta,0,s}} \| v \|_{X^{1/2,b_1,s}}\,,
\end{multline}
provided $1/4<b_1<3/8$.
 \end{lem}
 \begin{proof}
 By duality it is equivalent prove (\ref{estbil1}) with
\begin{equation*} 
 \Gamma^{\tau,\xi,q}_{\tau_1,\xi_1,q_1}:= \chi_{A_i} \chi_{\{|\xi_1|\le 1|, \, |\xi-\xi_1|\ge  2\}}
\langle \sigma \rangle^{-1/2}
\Bigl\langle\frac{\langle \sigma \rangle}
{\langle \xi \rangle^3}\Bigr\rangle^{b_1} |\xi| \langle \zeta \rangle^{s} \langle \zeta_1 \rangle^{-s} \langle \zeta-\zeta_1\rangle^{-s} \,.
\end{equation*}
 We thus proceed similarly to the previous propositions. The only difference is that here
  $ M_1$ describes the dyadic level $ 2^{-\N} $ and $M,M_2$ describe the dyadic level $ 2^{\N} $ .  
  Again, we only consider the case $s=0$.
  \\
We first prove (\ref{estLambda0}).   Note that  in $ A_0 $  we have $
|\xi_1|\le 1 $, $ |\xi|\sim |\xi_2|\ge 2 $  and
  $ |\sigma|\ge |\xi||\xi_1||\xi_2|$.  We can thus write
    $ K\sim 2^l M_1 M_2^2 $ with $ l\in \N $. We distinguish between the two
    cases  $ K<M^3$ and $  K\ge M^3$.\\
     i) $ K<M^3$ In this case we can write
  \begin{eqnarray*}
 I& \lesssim & \sum_{K_1,K_2 \atop M_1\lesssim 1,  M_2\gtrsim 1, l\in \N}(K_1\wedge K_2)^{1/2} \langle (K_1\vee K_2)^{1/4} M_1^{1/4}\rangle  \\
 & &
    \frac{M_1^{1/2} M_2}{2^{l/2} M_1^{1/2} M_2}
   \| f_{M_1,K_1}\|_{2}
   \| g_{M_2,K_2}\|_{2}   \| h_{M_2,2^l M_1 M_2^2}\|_{2}   \\
  & \lesssim & \sum_{K_1,K_2 \atop M_1\lesssim 1,  M_2\gtrsim 1 , l\in \N} 2^{-l/2} K_1^{3/8} K_2^{3/8}
  \| f_{M_1,K_1}\|_{2}
   \| g_{M_2,K_2}\|_{2}   \| h_{M_2,2^l M_1 M_2^2 }\|_{2}
 \end{eqnarray*}
  Summing in $ K_1 $, $ K_2 $ and then applying Cauchy-Schwarz in $ (M_1,M_2) $ implies
  \begin{eqnarray*}
 I& \lesssim & \sum_{ l\in \N} 2^{-l/2}
(\sum_{M_1}  \| f_{M_1}\|_{X^{\beta,0,0}}^2)^{1/2}  (\sum_{M_2}  \| g_{M_2}\|_{X^{\beta,0,0}}^2)^{1/2}
(\sum_{M_1,M_2}  \|  h_{M_2,2^l M_1 M_2^2}\|_{2}^2)^{1/2} \\
& \lesssim &  \| f \|_{X^{\beta,0,0}}   \| g \|_{X^{\beta,0,0}}   \| h \|_{2},
 \end{eqnarray*}
 provided $\beta\in (3/8,1/2)$.
 \\
  ii) $K\ge M^3 $. Then we have for $\beta\in (3/8,1/2)$,
\begin{eqnarray*}
 I & \lesssim & \sum_{ K_1,K_2 \atop M_1\lesssim 1, M_2\gtrsim 1, l\in \N}(K_1\wedge K_2)^{1/2} \langle(K_1\vee K_2)^{1/4}
 M_1^{1/4}\rangle\\
 && \frac{ M_1^{1/2} M_2^{1-3b_1}}{2^{(1/2-b_1)l} 
 M_1^{1/2-b_1} M_2^{1-2b_1} }
      \| f_{M_1,K_1}\|_{2}
   \| g_{M_2,K_2}\|_{2}   \| h_{M_2,2^lM_1M_2^2 }\|_{2}  \\
  & \lesssim & \sum_{K_1,K_2 \atop M_1\lesssim 1, M_2\gtrsim 1, l\in \N} 2^{-(1/2-b_1)l} (K_1 K_2)^{3/8-\beta} M_1^{b_1}M_2^{-b_1}
  \| f\|_{X^{\beta,0,0}}   \| g\|_{X^{\beta,0,0}}   \| h\|_{2}
  \end{eqnarray*}
  A direct summing in $ l, K_1, K_2, M_1 $ and $ M_2$ yields the desired result.
  \\
  
  Let us now prove (\ref{estLambda2}). In $ A_{2} $ ,
    we have $ |\xi_1|\le 1 $, $ |\xi|\sim |\xi_2|\ge 2 $  and $ |\sigma_2|\ge |\xi||\xi_1||\xi_2|
    >|\sigma| $.  We can thus
    write $ K_2\sim 2^l M_1 M_2^2 $ with $ l\in \N $. 
    \\
i) $ K<M^3$ In this case we can write  by using (\ref{pro1est1}) 
  \begin{multline*}
 I \lesssim  \sum_{K,K_1 \atop M_1\lesssim 1, M_2\gtrsim 1, l\in \N} (K\wedge K_1)^{1/2}  (K\vee K_1)^{1/4}  
 \\
    \frac{M_1^{1/2} M_2}{2^{l/2} M_1^{1/2} M_2  K^{1/2}}
   \| f_{M_1,K_1}\|_{2}
   (2^{l/2} M_1^{1/2} M_2)
   \| g_{M_2,2^l M_1 M_2^2}\|_{2}   \| h_{M_2,K}\|_{2}   
   \\
   \lesssim \sum_{K,K_1 \atop M_1\le M_2, l\in \N} 2^{-l/2} K ^{-1/8}K_1^{3/8}  
   \\
   \| f_{M_1,K_1}\|_{2}
  (2^{l/2} M_1^{1/2} M_2)
   \| g_{M_2,2^l M_1 M_2^2}\|_{2}   \| h_{M_2,K}\|_{2}
 \end{multline*}
  Summing in $ K $, $ K_1 $ and then applying Cauchy-Schwarz in $ (M_1,M_2) $, it results for $\beta\in(3/8,1/2)$,
  \begin{multline*}
 I \lesssim  \sum_{ l\in \N} 2^{-l/2}
(\sum_{M_1}  \| f_{M_1}\|_{X^{\beta,0,0}}^2)^{1/2}  (\sum_{M_2}  \| h
_{M_2}\|_{2}^2)^{1/2}
\\
(\sum_{M_1, M_2} 
 2^{l} M_1 M_2^2
 \|  g_{M_2,2^l M_1 M_2^2}\|_{2}^2)^{1/2} \lesssim   \| f \|_{X^{\beta,0,0}}  \| g \|_{X^{1/2,0,0}}   \| h \|_{2}
 \end{multline*}
  ii) $K\ge M^3 $. Then using that $ K_2 \ge K $ we have
$$
\frac{K/M^3}{K_2/M_2^3}\lesssim 1
$$
and thus we can proceed exactly as in the case $K<M^3$, the obtained bound being
$$
\| f \|_{X^{\beta,0,0}}  \| g \|_{X^{1/2,b_1,0}}   \| h \|_{2} \,.
$$  

Let us finally prove \eqref{R11}. In this case we have
$$
 K_1\gtrsim K+K_2,\quad K\lesssim M^3,\quad K_1\gtrsim M_1 M_2^2
$$
Thus using
(\ref{estbil1})  and \eqref{pro1est1} we can write
\begin{equation*}
 I  \lesssim  \sum_{ K,K_1,K_2 \atop M_1\lesssim 1, M_2\gtrsim 1}
 K_2^{1/2} K_1^{1/4}
  \frac{ M_1^{1/2} M_2}
 {K^{1/2}}\| f_{M_1,K_1}\|_{2}
   \| g_{M_2,K_2}\|_{2}   \| h\|_{2}  
   \end{equation*}
   (we simply neglect the $h$ localizations).
   We now naturally split the sum into two contributions, the contribution of $M_1^{1/2}M_2\geq 1$ and $M_1^{1/2}M_2\leq 1$.
   In the first case we can write for $\delta\in (0,b_1-1/4)$ and $\varepsilon>0$ such that $b_1-\delta-1/4-2\varepsilon>0$,
 \begin{multline*}
  I\lesssim \hspace*{-10mm} \sum_{ K,K_1,K_2 \atop M_1\lesssim 1, M_2\gtrsim 1, M_1M_2^2\geq 1}
 \hspace*{-10mm}  K_1^{1/4}
  \frac{ M_1^{1/2} M_2}
 {K^{1/2}K_1^{1/2+b_1-2\varepsilon}}
      K_1^{1/2+b_1-\varepsilon}\| f_{M_1,K_1}\|_{2}
   K_2^{1/2-\varepsilon}\| g_{M_2,K_2}\|_{2}   \| h\|_{2}  
   \\
   \lesssim
   \sum_{ K_1,K_2 \atop M_1\lesssim 1, M_2\gtrsim 1, M_1M_2^2\geq 1}
  \frac{ M_1^{1/2} M_2}
 {K_{1}^{1/2+\delta}K_1^{b_1-\delta-1/4-2\varepsilon}}
    K_1^{1/2+b_1-\varepsilon}   \| f_{M_1,K_1}\|_{2}
   K_{2}^{1/2-\varepsilon}\| g_{M_2,K_2}\|_{2}   \| h\|_{2}  
   \\
   \lesssim
  \sum_{  M_1\lesssim 1, M_2\gtrsim 1, M_1M_2^2\geq 1}
  \frac{ 1} {(M_1M_2^2)^{\delta}}
  \| f_{M_1}\|_{X^{1/2,b_1,0}} 
  \| g_{M_2}\|_{X^{\beta,0,0}}   \| h\|_{2}    
\\
\lesssim
\sum_{  M_2\gtrsim 1, l\in\N}
  2^{-\delta l}\| f_{2^{l}M_2^{-2}}\|_{X^{1/2,b_1,0}  } \| g_{M_2}\|_{X^{\beta,0,0}}   \| h\|_{2}    
  \lesssim
   \| f \|_{X^{1/2,b_1,0}}
   \| g \|_{X^{\beta,0,0}}   \| h \|_{2},
\end{multline*}
where  $ \beta\in ]1/2-\varepsilon,1/2[$ and in the last inequality we use the Cauchy-Schwarz inequality in the $M_2$ summation (at fixed $l$). 
In the second case the argument is even easier since we do not use the lower bound on $K_1$. Namely, we can write for $\beta\in (3/8,1/2)$,
\begin{multline*}
   \sum_{ K,K_1,K_2 \atop M_1\lesssim 1, M_2\gtrsim 1, M_1M_2^2\leq 1}
 K_2^{1/2} K_1^{1/4}
  \frac{ M_1^{1/2} M_2}{K^{1/2}}
      \| f_{M_1,K_1}\|_{2}
   \| g_{M_2,K_2}\|_{2}   \| h\|_{2}  
   \\
   \lesssim
\sum_{  M_2\gtrsim 1, l\in\N}
  2^{- l/2}\| f_{2^{-l}M_2^{-2}}\|_{X^{\beta,0,0}} \| g_{M_2}\|_{X^{\beta,0,0}}   \| h\|_{2}    
  \lesssim
   \| f \|_{X^{\beta,0,0}}   \| g \|_{X^{\beta,0,0}}   \| h \|_{2}\,.
\end{multline*}
This competes the proof of Lemma~\ref{lemme3}.
\end{proof}
Now   (\ref{bilinearestimate1}) follows by  combining Lemmas \ref{lemme1}-\ref{lemme3}.
\subsection{ Proof of (\ref{bilinearestimate2}).}\label{32}
 Since $X^{-1/2+,0,s} $ is continuously embedded in $ Z^s $ and since in the proof of
Lemmas \ref{lemme1}-\ref{lemme3}, except in the case {\bf A} in Lemma
\ref{lemme1} and in the proof of \re{estLambda0}
 in Lemma  \ref{lemme3}, we can keep a factor $ K^{0+} $ with $
h_K $, it remains to treat the corresponding regions. Actually it
is obvious to see that  we can even restrict ourselves to the 
intersection of these regions with the region $ |\sigma|\ge 2 $.
By duality we have to prove
 $$
J:=\Bigl| \int_{\R^4} \sum_{(q,q_1)\in \Z^2} \chi_{|\xi_1|\le |\xi-\xi_1|} \Theta^{\tau,\xi,q}_{\tau_1,\xi_1,q_1}
f(\tau_1,\xi_1,q_1) g(\tau_2,\xi_2,q_2) h(\xi,q) \, d\tau \,
d\tau_1 \, d\xi \, d\xi_1  \Bigr|
 $$
 $$
\lesssim \|f\|_{L^2(\R^2\times \T)} \|g\|_{L^2(\R^2\times \T)}
\|h\|_{L^2(\R\times \T)}
$$
where
\begin{equation*} 
 \Theta^{\tau,\xi,q}_{\tau_1,\xi_1,q_1}:= 
 \langle \sigma \rangle^{-1}|\xi| \langle \zeta \rangle^{s} \langle \zeta_1 \rangle^{-s} \langle \zeta-\zeta_1\rangle^{-s} \,.
 \end{equation*}
Again we can suppose that $s=0$ and moreover we will not make use of the factors involving $b_1$.
Recall that in the region $ A$,  of Lemma~\ref{lemme1}  we have $
M, M_1, M_2 \ge 1 $ and we can write $ K\sim 2^l M M_1 M_2 $ with
$ l\in \N $. We separate the cases $ M_1\sim M_2 $ and $ M\sim
M_2 $.
\begin{enumerate}
\item[{\bf A.1.}] $ K \gtrsim M M_1 M_2 $ and $M\lesssim M_1\sim M_2
 $.\\
Then  we can write $ K\sim 2^l M M_2^2 $ with $ l\in \Z_+ $.
 From (\ref{resonant}) and (\ref{coro2est2}) we thus have
 \begin{eqnarray*}
J & \lesssim & \sum_{K_1,K_2 \atop M\lesssim M_2, l\in \N} (K_1
K_2)^{\beta_0} \frac{M_2^{1/2}M  }{2^{l} M M_2^2 } \| f_{M_2,K_1}\|_{2}
   \| g_{M_2,K_2}\|_{2}   \| h_{M}\chi_{\{\langle \sigma\rangle \sim 2^l M M_2^2\}} \|_{2} \\
  & \lesssim & \sum_{K_1,K_2 \atop  M\lesssim M_2, l\in \N} 2^{-l/2} (K_1 K_2)^{\beta_0} \Bigl( \frac{M}{M_2}\Bigr)^{1/2}
   \| f_{M_2,K_1}\|_{2}
   \| g_{M_2,K_2}\|_{2}   \| h_{M}\|_{2} \\
 \end{eqnarray*}
 Summing over $ M\le M_2 $, $ K_1 $, $K_2 $ and  $l$ and then
 applying Cauchy-Schwarz in $ M_2 $ we get the desired result.
\item[{\bf A.2.}] $ K \gtrsim M M_1 M_2 $ and $M_1\lesssim M\sim M_2$. We have
 \begin{eqnarray}
J & \lesssim & \sum_{K_1,K_2 \atop  M_1\le M_2, l\in \N}  \frac{
M_2 }{2^{l} M_1 M_2^2 }
  \Bigl( h_{M_2}\chi_{\{\langle \sigma\rangle \sim 2^l M_1 M_2^2\}} ,
     f_{K_1,M_1} \star g_{K_2,M_2}\Bigr)_{L^2} 
     \label{yo}
 \end{eqnarray}
 If $K_2\geq K$, we can easily conclude by using (\ref{coro2est2}). We shall therefore suppose that $K\geq K_2$. 
For fixed $ (\xi,q) $,
$$
\chi_{\{\langle \sigma \rangle \sim K \}}\lesssim \chi_{\{\langle
\sigma \rangle \sim K \}} {\star_\tau} \Bigl(\frac{1}{K}
\chi_{\{|\tau| \le K \}}\Bigr) \; .
$$
Using that the function $\chi_{\{|\tau| \le K \}}$ is pair, the $L^2 $ scalar product in (\ref{yo}) can be estimated by 
$$
\Bigl( h_{M_2}\chi_{\{\langle \sigma\rangle \sim 2^l M_1 M_2^2\}}
,
     f_{K_1,M_1} \star \Bigl(g_{K_2,M_2}\star_\tau
      (\frac{1}{K}
\chi_{\{\langle \sigma \rangle \le  2^l M_1 M_2^2\}}) \Bigr)
\Bigr)_{L^2} \; .
$$
We have that $ g_{K_2,M_2}\star_\tau
      (\frac{1}{K}
\chi_{\{\langle \sigma \rangle \le  2^l M_1 M_2^2\}} $ is of the
form $ g'_{M_2,K+K_2}  $ with
\begin{equation} \label{cd}
\|g'_{M_2,K+K_2}\|_{2} \lesssim 
\Bigl(\frac{K_2}{K}\Bigr)^{1/2} \|g_{K_2,M_2}\|_{2} \; .
\end{equation}
Indeed the linear operator $ T_{K,K_2}\, :\, v\mapsto \frac{1}{K} v (\cdot)  \chi_{\{\langle \cdot \rangle \sim  K_2\}}\star \chi_{\{\langle \cdot \rangle \le  K\}}$
 is a continuous endomorphism  of $ L^1(\R) $ and of $ L^\infty(\R) $ with
 \begin{eqnarray*}
 \| T_{K_,K_2}v\|_{L^\infty (\R)}  &\le & \sup_{x\in\R} \frac{1}{K} \Bigl|  \int_{\R} v(y) \chi_{\{\langle y \rangle \sim  K_2\}}\star \chi_{\{\langle x-y \rangle \le  K\}}
 \, dy \Bigr| \\
 & \lesssim &  \frac{K_2}{K} \|v\|_{L^\infty (\R)}
 \end{eqnarray*}
 and
 $$
  \| T_{K_,K_2}v\|_{L^1 (\R)}  \le \frac{1}{K} \| v\|_{L^1(\R)} \|  \chi_{\{\langle \cdot \rangle \le  K\}}\|_{L^1(\R)} \lesssim \|v \|_{L^1(\R)} \; .
 $$
 Therefore, by Riesz interpolation theorem $ T_{K,K_2} $ is a continuous endomorphism of $ L^2(\R) $ with
 $$
  \| T_{K_,K_2}v\|_{L^2 (\R)}  
\Bigl(\frac{K_2}{K}\Bigr)^{1/2}  \| v\|_{L^2(\R)} \; .
 $$
 Applying (\ref{coro2est1})  with \re{cd} at hand we get
 \begin{multline*}
J  \lesssim  \sum_{K_1,K_2 \atop  M_1\le M_2, l\in \N}
K_1^{\beta_0}K_2^{\beta_0}  \frac{ M_1^{1/2} M_2 }{2^{l}M_1 M_2^2 } 
 \| f_{M_1,K_1}\|_{2}
   \| g'_{M_2,K+K_2}\|_{2}   \| h_{M_2}\chi_{\{\langle \sigma\rangle \sim 2^l M_1 M_2^2\}}\|_{L^2_{t,x,y}} 
   \\
     \lesssim  \sum_{K_1,K_2 \atop M_1\le M_2, l\in \N}  2^{-l/2} K_1^{\beta_0} K_2^{\beta_0} 
    \Bigl(\frac{K_2}{2^{l}M_1 M_2^2}\Bigr)^{1/2}    \| f_{M_1,K_1}\|_{2}
   \| g_{M_2,K_2}\|_{2}   \| h_{M_2}\|_{L^2_{x,y}} 
   \\
      \lesssim   2^{-l/2}K_1^{\beta_0} K_2^{\beta_0} K_2^\delta (M_1 M_2^2)^{-\delta}\| f_{M_1,K_1}\|_{2}
   \| g_{M_2,K_2}\|_{2}   \| h_{M_2}\|_{L^2_{x,y}} \; .
 \end{multline*}
 By choosing $\delta\in (0,1/2-\beta_0)$, we can
 sum over $ K_1,$ $K_2$, $M_1$, $M_2 $ and $ l$ which yields the
 needed bound.
\end{enumerate}
It remains to treat  the part of the region $ A_0 $ where $|\xi_1|\le 1 $, $|\xi_2|\ge 2$. We can write
   $$
   J\lesssim \sum_{K, K_1,K_2 \atop  M_1\lesssim 1, M_2\gtrsim1}  K^{-1}M_2 ( h_{M_2}\chi_{\{\langle \sigma\rangle \sim K\}} ,
     f_{K_1,M_1} \star g_{K_2,M_2})_{L^2}
   $$
 We now consider separately the contributions of the regions $M_1 M_2^2\leq 1$ and $M_1 M_2^2>1$ to $J$.
 Let us denote by $J_1$ the contribution of $M_1 M_2^2\leq 1$.
 Then we can write by using (\ref{coro2est1}) for some $\beta_0<\beta<1/2$ 
 \begin{eqnarray*}
 J_1 & \lesssim &
  \sum_{K, K_1,K_2 \atop  M_1\lesssim 1, M_2\gtrsim1, M_1 M_2^2\leq 1}  \frac{
M_2 M_1^{1/2}}{K^{1/2}}
(K_1 K_2)^{\beta_0}
  \|h_{M_2}\|_{2}\|f_{K_1,M_1}\|_{2}\|g_{K_2,M_2}\|_{2}
  \\
  & \lesssim &
  \sum_{M_1\lesssim 1, M_2\gtrsim1, M_1 M_2^2\leq 1}
  M_2 M_1^{1/2}\|h\|_{2}\|f_{M_1}\|_{X^{\beta,0,0}}  \|g_{M_2}\|_{X^{\beta,0,0}}
  \\
  & \lesssim &
  \sum_{l\in\N, M_2\gtrsim1}2^{-l/2}\|h\|_{2}\|f_{2^{-l}M_2^{-2}}\|_{X^{\beta,0,0}}  \|g_{M_2}\|_{X^{\beta,0,0}}
  \\
   & \lesssim &\|f\|_{X^{\beta,0,0}}\|g\|_{X^{\beta,0,0}}\|h\|_{2}\,,
   \end{eqnarray*}
 where in the last inequality we used the Cauchy-Schwarz inequality for fixed $l$ and then summing in a straightforward way in $l$.
 
 We next estimate the contribution of $M_1 M_2^2>  1$. 
 Denote by $J_2$ this contribution to $J$. Since
  $ |\sigma|\gtrsim |\xi||\xi_1||\xi_2|$, we can  write
    $ \langle \sigma\rangle \sim 2^l M_1 M_2^2 $ with $ l\in \N $.     
    We only consider the case $K\gtrsim K_2$, the case $K\lesssim K_2$ being simpler. Using \eqref{pro1est1} and proceeding 
     as in A.2 above, we obtain 
  \begin{multline*}
 J_2  \lesssim  \sum_{K_1,K_2\atop M_1\lesssim 1,  M_2\ge 1, l\in \N, M_1M_2^2>1}(K_1\wedge K_2)^{1/2} \langle (K_1\vee K_2)^{1/4} M_1^{1/4}\rangle  
 \\
    \frac{M_1^{1/2} M_2}{2^{l} M_1 M_2^2 }
   \| f_{M_1,K_1}\|_{2}
   \| g'_{M_2,K+K_2}\|_{2}   \| h_{M_2} \chi_{\{\langle \sigma\rangle \sim 2^l M_1 M_2^2\}}\|_{2} ,  
    \end{multline*}
   where again $g'_{M_2,2^lM_1M_2^2+K_2}$satisfies \eqref{cd}. Thus
   \begin{multline*}
 J_2  \lesssim  \sum_{K_1,K_2, 2^lM_1M_2^2\gtrsim K_2\atop M_1\lesssim 1,  M_2, l\in \N, M_1M_2^2>1}(K_1\wedge K_2)^{1/2} \langle (K_1\vee K_2)^{1/4} M_1^{1/4}\rangle  
 \\
    \frac{M_1^{1/2} M_2}{2^{l} M_1 M_2^2   }
    \left(    \frac{K_2}{2^lM_1M_2^2}\right)^{1/2}
  (2^lM_1M_2^2)^{1/2}
   \| f_{M_1,K_1}\|_{2} \| g_{M_2,K_2} \|_{2}   \| h_{M_2}\|_{2} ,  
    \end{multline*}
   We therefore obtain that for a suitable $\delta>0$ and $\beta<1/2$,
   \begin{eqnarray*}
   J_2 & \lesssim &
   \sum_{M_1\lesssim 1,  M_2, M_1M_2^2>1}
   (M_1 M_2^2)^{-\delta}   \| f\|_{X^{\beta,0,0}} \| g_{M_2}\|_{X^{\beta,0,0}}   \| h_{M_2}\|_{2} 
   \\
&   \lesssim & \|f\|_{X^{\beta,0,0}}\|g\|_{X^{\beta,0,0}}\|h\|_{2} \,,
   \end{eqnarray*}
  where in the last inequality we write $ M_1 M_2^2=2^q $ with $ q\in \N $, 
we apply the  Cauchy-Schwarz in $ M_2$ and then we
sum in $ q$. This completes the proof of Proposition~\ref{probilest_bis}.
  \qed 
\subsection{Global well-posedness}\label{33}
Let $ s\ge 0 $ and $ 1/4<b_1<3/8 $. Noticing that in all the estimates in Subsections  \ref{31}-\ref{32} we can  keep a factor
$ K_1^{0+} $ or $ K_2^{0+} $ with $ f_{M_1,K_1} $ or $ g_{M_2,K_2} $ 
and that for any functions $ v\in  X^{1/2,s} $ supported in time  in
$ ]-T,T[ $,
\begin{equation}
\|v\|_{X^{\beta,0,s}} \lesssim T^{(1/2-\beta)-} \|v\|_{X^{1/2,0,s}} \;
,\label{esT}
\end{equation}
we infer that  the following restricted bilinear estimates hold :
 \begin{equation} \label{restbilinearestimate1}
 \|\partial_x(uv) \|_{X^{-1/2,b_1,s}_T} \lesssim
 T^\nu\|u\|_{X^{1/2,b_1,s}_T}\|v\|_{X^{1/2,b_1,s}_T}
 \end{equation}
 and
  \begin{equation} \label{restbilinearestimate2}
 \|\partial_x(uv) \|_{Z^{-1/2,s}_T} \lesssim T^\nu \|u\|_{X^{1/2,b_1,s}_T}\|v\|_{X^{1/2,b_1,s}_T}
 \, ,
 \end{equation}
 with $ \nu>0 $.
 Estimates \eqref{restbilinearestimate1}, \eqref{restbilinearestimate2} in conjugation with the linear estimate (see \cite{Gi})
 $$
\Big \|\psi(t)\int_0^t U(t-t') \partial_x F(t')dt'\Big\|_{ X^{1/2,b_1,s}\cap  Z^{1/2,s}}\lesssim
 \|F\|_{ X^{-1/2,b_1,s}\cap  Z^{-1/2,s}},
 $$
 with $ \psi\in C^\infty_0(\R)$,
 lead to the well-posedness result by a standard fixed point argument
  in $ X^{1/2,b_1,s}_T \cap Z^{1/2,s}_T$, $ T>0 $ small enough,  on the Duhamel formulation of \re{KP2} :
  \begin{equation}\label{duha}
  u(t)=U(t)\varphi-\frac{1}{2}\int_0^t U(t-t') \partial_x (u^2(t')) \, dt'
  \end{equation}
  Also standard considerations prove that the time of existence $ T^* $ of the solution
   only depends on $\|\varphi\|_{L^2}$ (see \eqref{tame}). 
The uniqueness statement for $s>2$ follows from the Gronwall lemma.
Next, from \re{duha} and the fact that
    $H^s(\R\times\T)  $ is an algebra for $ s>1 $ we infer that for $ s>1 $,
     $\partial_x^{-1} u$ belongs to $  C([0,T^*];H^s(\R\times\T)) $ provided $ \varphi\in H^s_{-1}(\R\times \T) $ where 
     $$
\|\varphi\|^2_{H^{s}_{-1}}=\sum_{\eta\in \Z}\int _{\R_{\xi}}
\langle|\xi|^{-1} \rangle^2
\langle(\xi,\eta)\rangle^{2s}|\hat{\varphi}(\xi,\eta)|^2 d\xi \,  .$$
      Taking the $ L^2$ scalar product of \re{KP2} with $\partial_x^{-1} u $ it is then
       easy to check that the $ L^2(\R\times\T)$-norm of such solutions is a constant of the motion.
      The density of $H^s_{-1}(\R\times \T) $ in $L^2(\R\times \T) $ combining with the
       continuity with respect to initial data in $ L^2(\R\times\T) $ ensures that the $ L^2 $-norm is a constant of the motion for our solutions (see \cite{mol} for details
       on this point). 
       This proves that the solutions exist for all time.
This completes the proof of Theorem~\ref{theo1}.

\section{Proof of Theorem \ref{theo2}}
 We write the solution $u$ of
(\ref{KP2}), (\ref{1.2}) as
$$
u(t,x,y)=\psi_c(x-ct,y)+v(t,x-ct,y)
$$
where $v(t,.,.) $ is a  $ L^2(\R^2)$-function. This $v$ satisfies the
equation
 \begin{equation}
(v_t-cv_x+v_{xxx} +v v_x +\partial_x (\psi_c v) )_x - v_{yy} =0,\quad
v(0,x,y)=\phi(x,y). \label{LineKP}
\end{equation}
Our strategy is to perform a fixed point argument in some
Bourgain's type spaces on the Duhamel formulation of \re{LineKP}.
In the context of \eqref{LineKP} some straightforward modifications taking into account the term $cv_x$ of the unitary group $U(t)$ and the Bourgain spaces should be done.  

 We will use in a crucial way the following linear estimates of Strichartz or smoothing type  (see for instance \cite{Sa}) injected in the framework of
the Bourgain spaces.
\begin{lem}\label{lem41}
One has
\begin{equation}\label{strichartz}
\|v\|_{L^4_{txy}} \lesssim \|v\|_{X^{b,0,0}}
\end{equation}
and
\begin{equation} \label{kato}
\|v_x\|_{L^\infty_x L^2_{ty}} \lesssim \|v\|_{X^{b,0,0}}\,,
\end{equation}
provided $b>1/2$.
\end{lem}
\begin{proof}
We only brievly recall  the proof of \eqref{kato}, for \eqref{strichartz} we refer to \cite{TT} for instance.
By usual considerations (see \cite{Gi}) it suffices to prove that
\begin{equation} \label{u2}
\|\partial_x (U(t) \varphi )\|_{L^\infty_x L^2_{ty}}  \lesssim
\|\varphi\|_{L^2_{xy}}\; .
\end{equation}
We first notice that
$$
\partial_x (U(t) \varphi )(x,y)= c \int_{\R^2} (i\xi) e^{i(x\xi+y
\eta)} e^{i(\xi^3+c\xi-\eta^2/\xi)t} \hat{\varphi}(\xi,\eta) \, d\xi \,
d\eta
$$
and that, for any fixed $\eta $, $ \xi\mapsto \xi^3+c\xi-\eta^2/\xi $ is an
increasing  bijection from  $ \R_+^* $ into $ \R $ and from  $
\R_-^* $  into $ \R $.
Therefore splitting the $\xi$ integral into two pieces corresponding to positive and negative values of $\xi$ and after 
performing the change of variables $
\xi\mapsto \xi^3+c\xi-\eta^2/\xi $ a short computaion leads to
$$
\|\partial_x U(t) \varphi \|_{L^2_{ty}}^2=c\int_{\R^2}
|\hat{\varphi}(\xi,\eta)|^2 \frac{\xi^2}{3\xi^2+c+\eta^2/\xi^2} \,
d\xi d\eta.
$$
This yields (\ref{u2}).
\end{proof}
\begin{lem}\label{lem42}
For any $ s\ge 0 $, $ 0<T\leq 1 $, $b>1/2$ one has
\begin{equation}\label{eqd}
\|\partial_x(\psi_c v) \|_{X^{0,0,s}_T} \lesssim 
\Bigl(\|\partial_x \psi_c \|_{W^{s,\infty}}+ \sum_{|\alpha|\leq s }\| \partial^{\alpha}_{x,y}\psi_c\|_{L^2_{x} L^\infty_{y}} \Bigr)\|v\|_{X^{b,0,s}_T}
\end{equation}
\end{lem}
\begin{proof}
We write $ \partial_x (\psi_c v) $ as $
\partial_x \psi_c v +v_x \psi_c $ and we treat each term separately. Write
\begin{eqnarray*}
\|\partial_x \psi_c v\|_{X^{0,0,s}_T} & = &  \|\partial_x
\psi_c v \|_{L^{2}_T H^s_{xy}} 
\\
& \lesssim &
\|\partial_x \psi_c\|_{W^{s,\infty}}
\|v\|_{L^\infty_T H^s}
\\
& \lesssim &  \|\partial_x  \psi_c\|_{W^{s,\infty}}
\|v\|_{X^{b,0,s}_T}
\end{eqnarray*}
Further, we have
\begin{equation*}
\|\psi_c v_x\|_{X^{0,0,s}_T}
\lesssim 
(\sum_{|\alpha|\leq s }\| \partial^{\alpha}_{x,y}\psi_c\|_{L^2_{x} L^\infty_{Ty}})
(\sum_{|\beta|\leq s }\| \partial^{\beta}_{x,y} v_x\|_{L^\infty_x L^2_{ty}})
\end{equation*}
This leads to (\ref{eqd}) thanks to Lemma \ref{lem41}
\end{proof}
\begin{pro}\label{probilest_tris}
There exists $ \varepsilon_0>0 $ such that the following bilinear estimate holds
 \begin{equation} \label{bilinearestimate3}
 \|\partial_x(uv) \|_{X^{-1/2+2\varepsilon,b_1,s}} \lesssim
 \|u\|_{X^{1/2+\varepsilon,b_1,s}} \|v\|_{X^{1/2+\varepsilon,b_1,s}} \;
 \end{equation} 
provided $ 1/4<b_1<3/8 $, $ s\ge 0 $ and $ 0<\varepsilon< \varepsilon_0 $.\\
 \end{pro}
\begin{proof}
We have that Proposition~\ref{propo1}, Corollaries~\ref{coro1}-\ref{coro2} and their
 proofs are still valid for functions of $ \R^3 $. Since as noticed
 in Subsection~\ref{32} we can always keep a factor $ K $ except in the cases
 $ |\sigma|\ge 2  $ dominant  and $ |\xi|\ge 2 $, we deduce that there exists $ \varepsilon_0>0 $ such that
\begin{multline*}
\Big\| {\mathcal F}_{txy}\Bigl(\chi_{\{ |\sigma|<\max(\langle \sigma_1 \rangle, \langle \sigma_2 \rangle)\}}  \chi_{\{|\xi|<2\}}\xi\,
 \widehat{u}\ast\widehat{v} \Bigr)\Bigr\|_{X^{-1/2+2\varepsilon,b_1,s}}
\\
\lesssim
 \|u\|_{X^{1/2+\varepsilon,b_1,s}} \|v\|_{X^{1/2+\varepsilon,b_1,s}} \;
\end{multline*}
 provided $ 1/4<b_1<3/8 $, $ s\ge 0 $ and $ 0<\varepsilon< \varepsilon_0 $.
It thus remains to
 treat this region $ |\sigma|\ge \max(\langle \sigma_1 \rangle, \langle \sigma_2 \rangle) $, $ |\xi|\ge 2 $. We thus have to
 prove that
\begin{multline*}
J:=\Bigl| \int_{\R^6}
\Gamma^{\tau,\xi,q}_{\tau_1,\xi_1,\eta_1} f(\tau_1,\xi_1,\eta_1)
g(\tau_2,\xi_2,\eta_2) h(\tau,\xi,\eta) \, d\tau \, d\tau_1 \, d\xi \,
d\xi_1 d\eta d\eta_1 \Bigr|
\\
\lesssim \|f\|_{2} \|g\|_{2} \|h\|_{2}
\end{multline*}
where
\begin{multline}\label{defGamma}
 \Gamma^{\tau,\xi,\eta}_{\tau_1,\xi_1,\eta_1}:=  \chi_{\{1\le
|\xi_1|\le |\xi-\xi_1|, \, |\xi|\ge 2, \, |\sigma|\ge
\max(\langle \sigma_1 \rangle, \langle \sigma_2\rangle)\}}
\\  \frac{\Bigl\langle \frac{\langle \sigma
  \rangle}{\langle \xi \rangle^3}\Bigr\rangle^{b_1} |\xi| \langle \zeta \rangle^{s} \langle \zeta_1 \rangle^{-s} \langle \zeta-\zeta_1
  \rangle^{-s} }{
  \langle \sigma \rangle^{1/2-2\varepsilon}
  \langle \sigma_1 \rangle^{1/2+\varepsilon}\Bigl\langle \frac{\langle \sigma_1
  \rangle}{\langle \xi_1 \rangle^3}\Bigr\rangle^{b_1}
  \langle \sigma_2 \rangle^{1/2+\varepsilon}\Bigl\langle \frac{\langle
  \sigma_2 \rangle}{\langle \xi-\xi_1 \rangle^3}\Bigr\rangle^{b_1} }
\end{multline}
with $ \sigma:=\sigma(\tau,\xi,\eta):=\tau-\xi^3-c\xi+\eta^2/\xi $,
 $ \sigma_1:=\sigma(\tau_1,\xi_1,\eta_1)$ and 
$$
 \sigma_2:=\sigma(\tau-\tau_1,\xi-\xi_1,\eta-\eta_1).
$$
Notice that the crucial non resonance relation \eqref{resonant} for $ \sigma$, $ \sigma_1$ and $ \sigma_2$ still holds with the slight modification of the definition of $\sigma$. 
 We can of course assume that $ f$, $ g$  and $ h$ are
 non-negative functions in $ \R^3 $. We separate the domain of integration  into two regions. 
\\
   {\bf 1. $ 100 |\sigma|\geq |\xi|^3$.}   By Plancherel and then H\"{o}lder inequality,  using the Strichartz inequality (\ref{strichartz})
    , we infer that for $\varepsilon\ll 1$,
  \begin{multline*}
  J  \lesssim 
 \int_{\R^3} |\xi|^{1-3b_1-3(1/2-2\varepsilon-b_1)}h(\tau,\xi,\eta) 
\\ 
\int_{\R^3} 
\frac{f(\tau_1,\xi_1,\eta_1)g(\tau-\tau_1,
  \xi-\xi_1,\eta-\eta_1)}{\langle \sigma_1\rangle^{1/2+\varepsilon} \langle
  \sigma_2\rangle^{1/2+\varepsilon}}d\tau_1 \, d\xi_1 \, d\eta_1  d\tau d\xi d\eta
\\
\lesssim 
\Bigl\|{\mathcal F}_{\tau,\xi,\eta}\Bigl( \frac{f}{\langle
   \sigma\rangle^{1/2+\varepsilon}}\Bigr)\Bigr\|_4
  \Bigl\|{\mathcal F}_{\tau,\xi,\eta}\Bigl( \frac{g}{\langle \sigma
  \rangle^{1/2+\varepsilon}}\Bigr)\Bigr\|_{L^4} \|h\|_{L^2}
\lesssim  \|f\|_{L^2}\|g\|_{L^2}\|h\|_{L^2}
  \end{multline*}
  {\bf  2.  $   100 |\sigma|\leq |\xi|^3    $.} 
Applying  Cauchy-Schwarz in $
(\tau_1,\xi_1,\eta_1) $  and setting $(\tau_2,\xi_2,\eta_2):=
(\tau-\tau_1,\xi-\xi_1,\eta-\eta_1) $ we get
  \begin{multline}\label{yy}
  J \lesssim  \int_{\R^3} I(\tau,\xi,\eta)
\\
 \Bigl[\int_{\R^3}
 \Bigl| f(\tau_1,\xi_1,\eta_1)g(\tau_2,\xi_2,\eta_2)\Bigr|^2 \,
  d\tau_1 \, d\xi_1 \, d\eta_1\Bigr]^{1/2} h(\tau,\xi,\eta)
  d\tau d\xi d\eta , 
  \end{multline}
  where, using the elementary inequality ,
$$
\int_{\R} \frac{d\theta}{\langle \theta-a\rangle^{1+2\varepsilon} \langle\theta-b \rangle^{1+2\varepsilon}} \, \lesssim
 \frac{1}{\langle a+b\rangle^{1+2\varepsilon}}
$$
it holds, by the resonance relation, 
  \begin{eqnarray*}
  I(\tau,\xi,\eta) &\lesssim &  \frac{|\xi|}  {\langle
  \sigma\rangle^{1/2-2\varepsilon}}
   \Bigl[
  \int_{\R^2} \Bigl(\int_{\max(|\sigma_1|,|\sigma_2|) \le |\sigma|}\frac{d\tau_1}{\langle \sigma_1 \rangle^{1+2\varepsilon}   \langle \sigma_2 \rangle^{1+2\varepsilon}
} \Bigr)  d\xi_1 d\eta_1 \Bigr]^{1/2}
 \\
  & \lesssim &  \frac{|\xi| }{\langle
  \sigma\rangle^{1/2-2\varepsilon}}\Bigl(\int_{\R^2\cap \{|\xi\xi_1(\xi-\xi_1)|\le |\sigma|\}} \frac{d\xi_1 \,
   d\eta_1}{\langle \sigma_1+\sigma_2\rangle^{1+2\varepsilon} }\Bigr)^{1/2}
  \end{eqnarray*}
  We perform the change of variables $ (\xi_1,\eta_1)\mapsto
  (\nu,\mu) $ with 
$$
  \left\{\begin{array}{l}
  \nu=3\xi\xi_1(\xi-\xi_1)\\
  \mu=\sigma_1+\sigma_2
  \end{array}
  \right.
$$
Noticing that $ \nu\in [-3|\sigma|, 3|\sigma|]$, 
$$
d\xi_1 \, d\eta_1= c \frac{|\nu|^{1/2} \, d\nu\,
d\mu}{|\xi|^{3/2} (\frac{3}{4} \xi^3-\nu)^{1/2}
|\sigma+\nu-\mu|^{1/2}}
$$
and using the elementary inequality
$$
\int_{\R} \frac{ d\theta}{\langle \theta-a\rangle^{1+2\varepsilon}
 |\theta-b |^{1/2}}  \lesssim
 \frac{1}{\langle a+b\rangle^{1/2}}\, ,
$$
we thus infer that
\begin{eqnarray*}
I &\lesssim & \frac{|\xi|^{1/4} }{\langle\sigma\rangle^{1/2-2\varepsilon}}\Bigl[\int_{-3|\sigma|}^{3|\sigma|} \int_{-\infty}^{+\infty}
 \frac{|\nu|^{1/2} \, d\nu\, d\mu} { \langle \mu \rangle^{1+2\varepsilon}(\frac{3}{4} \xi^3-\nu)^{1/2} |\sigma+\nu-\mu|^{1/2}} \Bigr]^{1/2}\\
 & \lesssim & \frac{|\xi|^{1/4}}{\langle\sigma\rangle^{1/2-2\varepsilon}}\Bigl[\int_{-3|\sigma|}^{3|\sigma|}
 \frac{|\nu|^{1/2} \, d\nu} { (\frac{3}{4} \xi^3-\nu)^{1/2} \langle\sigma+\nu\rangle^{1/2}} \Bigr]^{1/2}\, .
\end{eqnarray*}
 By the definition of $ \nu $ it holds $ |\nu |\le \frac{3}{8}|\xi|^3 $, i.e.   $ | \frac{3}{4} \xi^3-\nu|\sim |\xi|^3 $ and thus
$$
I\lesssim
\frac{|\xi|^{-1/2}}{\langle\sigma\rangle^{1/2-2\varepsilon}}\Bigl[\int_{-3|\sigma|}^{3|\sigma|}
 \frac{|\nu|^{1/2} \, d\nu} { \langle\sigma+\nu\rangle^{1/2}} \Bigr]^{1/2}
 \lesssim |\xi|^{-1/2}\langle\sigma\rangle^{2\varepsilon}
 \lesssim |\xi|^{-1/2+6\varepsilon}\lesssim 1,
 $$
 provided $ 0<\varepsilon\le 1/12 $. This concludes the proof of \re{bilinearestimate3} by applying Cauchy-Schwarz in $ (\tau,\xi,\eta) $ in \re{yy}.
This completes the proof of Proposition~\ref{probilest_tris}.
\end{proof}
For $T\leq 1$ and $\varepsilon\ll 1$, , we have the following estimate (see \cite {Gi}, \cite{TT})
 \begin{equation}\label{TT1}
\Big\|
\int_0^t U(t-t')F(t')dt'
\Big\|_{X^{1/2+\varepsilon,b_1,s}_T}
\lesssim
\|F\|_{X^{-1/2+\varepsilon,b_1,s}_T}\,.
 \end{equation}
Moreover, using \cite[Theorem~3.1]{TT}, we obtain that for some $\nu>0$,
 \begin{equation}\label{TT2}
\|F\|_{X^{-1/2+\varepsilon,b_1,s}_T}
\lesssim
T^\nu\|F\|_{X^{-1/2+2\varepsilon,b_1,s}_T}\,.
\end{equation}
Combining Proposition \ref{probilest_tris}, Lemma \ref{lem42} and the bounds \eqref{TT1}, \eqref{TT2}, 
we infer that the map
  $$
  {\mathcal G} : v\mapsto  U(t) \phi -\int_0^t U(t-t') \partial_x
  (v^2/2+\psi_c v) \, dt'
  $$
  is a strictly contractive map in the ball of radius $ R:=2\|\phi\|_{H^s} $
  of $ X^{1/2+\varepsilon,b_1,s}_T $ provided $ T=T(R)>0 $ is
  small enough. Therefore there exists a unique local solution.
  Moreover, arguing  as in Subsection \ref{33} it is easy to check
  that the following differential identity holds for our solutions
  :
  $$
\frac{1}{2} \frac{d}{dt} \int_{\R^2} v^2 =-\frac{1}{2} \int_{\R^2}
\partial_x(\psi_c)\, v^2
  $$
and thus
$$
\|v(t)\|_{L^2} \lesssim \exp(t\|\psi_x\|_{L^\infty})
\|\phi\|_{L^2} \; .
$$
This leads to the global well-posedness result.
\begin{merci}
The Authors acknowledge support from the project ANR-07-BLAN-0250 of the Agence Nationale de la Recherche.
\end{merci}

\end{document}